\documentclass[a4paper,leqno]{article}

\usepackage[T1]{fontenc}
\usepackage[latin1]{inputenc}                                                                     
\usepackage{mathtools,amssymb,amsthm}
\mathtoolsset{mathic}
\usepackage{lmodern}
\usepackage{eucal}
\usepackage{tikz}
\usetikzlibrary{shapes,matrix,arrows}
\usepackage{hyperref}
\usepackage{microtype}
\usepackage{cleveref}

\theoremstyle{plain}
\newtheorem{thm}{Theorem}[subsection]

\newtheorem{coro}[thm]{Corollary}

\newtheorem{prop}[thm]{Proposition}

\newtheorem{defi}[thm]{Definition}
\newtheorem{theorem}{Theorem}
\newtheorem{corollary}{Corollary}

\theoremstyle{definition}
\newtheorem{rmrk}[thm]{Remark}
\theoremstyle{remark}

\newcommand{\R}{\mathbb{R}}
\newcommand{\Z}{\mathbb{Z}}
\newcommand{\N}{\mathbb{N}}
\renewcommand{\H}{\mathbb{H}}
\newcommand{\C}{\mathbb{C}}

\newcommand{\abs}[1]{\left|#1\right|}

\newcommand{\Q}{\mathbb{Q}}
\newcommand*{\rom}[1]{\romannumeral}
\renewcommand{\Re}{{Re}}

\providecommand{\msc}[1]{{\noindent\small\textbf{Mathematics Subject Classification (2020)} --- #1.}}
\providecommand{\keywords}[1]{{\noindent\small\textbf{Keywords} --- #1.}}

\DeclareMathOperator{\fl}{flat}

\DeclareMathOperator{\GL}{GL}
\DeclareMathOperator{\SO}{SO}

\DeclareMathOperator{\Spin}{Spin}

\DeclareMathOperator{\RE}{Re}
\DeclareMathOperator{\Tr}{Tr}
\DeclareMathOperator{\tr}{tr}
\DeclareMathOperator{\End}{End}
\DeclareMathOperator{\Vol}{Vol}
\DeclareMathOperator{\rank}{rank}
\DeclareMathOperator{\Id}{Id}
\DeclareMathOperator{\Ad}{Ad}
\DeclareMathOperator{\ad}{ad}
\DeclareMathOperator{\T}{T}
\DeclareMathOperator{\Norm}{Norm}

\DeclareMathOperator{\spec}{spec}

\DeclareMathOperator{\sign}{sign}

\DeclareMathOperator{\prim}{prime}

\DeclareMathOperator{\gr}{gr}
\DeclareMathOperator{\Rep}{Rep}
\DeclareMathOperator{\ev}{even}
\DeclareMathOperator{\im}{Im}
\DeclareMathOperator{\Ker}{Ker}
\DeclareMathOperator{\Flat}{Flat}
\DeclareMathOperator{\Mat}{Mat}

\numberwithin{equation}{section}

\renewcommand{\thethm}{%
\ifnum\value{subsection}=0
\thesection
\else
\thesubsection
	\fi
		.\arabic{thm}%
}

\providecommand{\contact}{{
	\bigskip
	\small

	\noindent
	\textsc{P.~Spilioti}, 
	Mathematisches Institut, Georg-August Universit\"at G\"ottingen, 
	Bunsenstrasse 3-5, D-37073 G\"ottingen, Germany
	\par\noindent\nopagebreak
	\textit{E-mail address:} \href{mailto:polyxeni.spilioti@mathematik.uni-goettingen.de}
	{\texttt{polyxeni.spilioti@mathematik.uni-goettingen.de}}}
}

\title{Twisted Ruelle zeta function on hyperbolic manifolds and complex-valued analytic torsion}
\author{Polyxeni Spilioti}

\begin{document}
\maketitle

\begin{abstract}
In this paper, we study the twisted Ruelle zeta function
associated with the geodesic flow of a compact, hyperbolic, odd-dimensional manifold $X$.
The twisted Ruelle zeta function is associated with an acyclic representation $\chi\colon \pi_{1}(X) \rightarrow \GL_{n}(\C)$, which is close enough to an acyclic, unitary representation.  In this case, the twisted Ruelle zeta function is regular at zero and equals the square of the refined analytic torsion, as it is introduced by Braverman and Kappeler
in \cite{BK2}, multiplied by an exponential, which involves the eta invariant of the even part of the odd-signature operator, associated with $\chi$.
\end{abstract}

\bigskip
\newline
\msc{Primary: 37C30; Secondary: 58J52, 22E45, 11F72, 11M36, 43A85}
\tableofcontents

\section{Introduction}

The twisted dynamical zeta functions of Ruelle and Selberg are dynamical zeta functions, which are associated with the geodesic flow on the unit sphere bundle $S(X)$ of  a compact, hyperbolic manifold $X$. They are defined in terms of the lengths of the closed geodesics, also 
called length spectrum. The twisted dynamical zeta functions are defined by Euler-type products, which converge in some right half-plane of $\C$. The main goal of this paper is to prove that the twisted Ruelle zeta function associated with an acyclic representation of the fundamental group of $X$, which is close enough to an acyclic and unitary one, 
is regular at zero and moreover is equal to the Cappell-Miller torsion as it is introduced in 
\cite{cm}. For such representations, there exists a precise relation between the Cappell-Miller torsion and the refined analytic torsion, as it is introduced by Braverman and Kappeler in \cite{BK2}.
Hence, we conclude that the twisted Ruelle zeta function at zero is equal to the square of the refined analytic torsion multiplied by an exponential, which involves the eta invariant of the even part of the odd-signature operator associated with $\chi$.

\paragraph{Fried's conjecture and related work.}
Fried in \cite{Fried} proved that for a closed, oriented, hyperbolic manifold
$X=\Gamma\backslash \H^{d}$, and an acyclic, orthogonal representation $\varrho_{1}$ of
$\pi_{1}(S(X))$, the Ruelle zeta function, defined for Re$(s)>d-1$ by 
\begin{equation*}
 R(s;\varrho_{1}):=\prod_{\substack{[\gamma]\neq e,\\ [\gamma]\prim}}{\det(\Id-\varrho_{1}(\gamma)e^{-sl(\gamma)}}),
\end{equation*}
admits a meromorphic extension to $\C$,
and for $\epsilon=(-1)^{d-1}$
\begin{equation}
\vert R(0;\varrho)^{\epsilon}\vert=\tau_{\varrho}(S(X)),
\end{equation}
where $\tau_{\varrho_{1}}(S(X))\in\R^{+}$ is the Reidemeister torsion,
a topological invariant (\cite{reidemeister1935homotopieringe},
\cite{franz1935torsion}, \cite{DR}).

Fried in \cite{fried1987lefschetz} conjectured that the same result holds for 
all compact locally homogeneous Riemannian spaces $X$ and 
acyclic bundles over $S(X)$. This conjectured was proved by
\cite{MS91} and \cite{Shen2018} for negatively curve locally symmetric spaces.

Fried, in the same paper, dealt mostly with
hyperbolic manifolds of odd dimension.
Considering  an acyclic, orthogonal representation $\varrho$
of $\pi(X)$ and
using the Selberg trace formula for the heat operator $e^{-t\Delta_{j}}$, 
where $\Delta_{j}$ is the Hodge Laplacian  on $j$-forms on $X$
with values in the flat vector bundle  $E_{\varrho}$ associated with $\varrho$,
he proved the meromorphic continuation of the zeta functions to the whole complex plane, as well as functional equations for the Selberg zeta function (\cite[p. 531-532]{Fried}). 
In addition, he proved that  $R(0;\varrho)$ is regular at $s=0$ and
\begin{equation}\label{fried}
R(0;\varrho)=T_{\varrho}(X)^2,
\end{equation}
where $T_{\varrho}(X)$ is the Ray-Singer analytic torsion defined as in \cite{RS}
(see \cite[eq. (18)]{Fried}).
By Cheeger-M\"uller theorem (\cite{cheeger1979analytic}, \cite{muller1978analytic}), the analytic torsion
$T_{\varrho}(X)$ is equal to the Reidemeister torsion $\tau_{\varrho}(X)$.
Since $\pi_{1}(S(X))\cong \pi_{1}(X)$ and $\tau_{\varrho}(X)^{2}=\tau_{\varrho_{1}}(S(X))$ (\cite[p. 526]{Fried}),
one obtains (1.1).
In fact, Fried proved a more general result. If the representation 
$\varrho$ is not assumed acyclic, then by \cite[Theorem 3]{Fried}, 
the leading term in the Laurent expansion of $R(s;\varrho)$
at $s=0$ is 
\begin{equation*}
C_{\varrho}T_X(\varrho)^2 s^{e},
\end{equation*}
where $e$ is a linear combination of the twisted Betti 
numbers $\beta_{i}=\dim H^{i}(X;\varrho)$,
and $C_{\varrho}\in\Q$ is defined in therms of $\beta_{i}$.
The theorem of Fried gave rise to important applications,
such as the asymptotic behavior of the analytic torsion (\cite{M2}) 
and the study of the growth of the torsion in the cohomology of a closed arithmetic hyperbolic $3$-manifold (\cite{MMar}).

The connection of the Ruelle zeta function to spectral invariants such as the analytic 
torsion has been studied by Bunke and Olbrich (\cite{BO}), Wotzke (\cite{Wo}), 
M\"{u}ller (\cite{M2}) for compact hyperbolic manifolds, under certain assumptions on the representation of the fundamental group of the manifold.
For the case of a hyperbolic manifold of finite volume, we refer the reader to the work of 
Park (\cite{Park}) and Pfaff (\cite{pfaff2}, \cite{pfaff1}).
An advanced study of the dynamical zeta functions of locally symmetric manifolds of higher rank is  due to Moscovici and Stanton in \cite{MS91}, Deitmar in \cite{Deit1}, Shen in \cite{Shen2018} and Moscovici, Stanton and Frahm in \cite{MSF}. We mention also the work of Dang, Guillarmou, Riviere and Shen (\cite{dang2019fried})
where the Fried 's conjecture is treated for dimension $3$,  for Anosov
flows near the geodesic flow on the unit tangent bundle of hyperbolic $3$-manifolds.
This result was extended by Chaubet and Dang to higher dimensions in \cite
{chaubet2019dynamical}, where a new object is introduced, the dynamical torsion.
Ceki{\'c}, Delarue, Dyatlov and Paternain studied the 
Ruelle zeta function at zero for nearly hyperbolic 3-manifolds in \cite{cekic2022ruelle}.
The twisted Ruelle zeta function associated with a non-unitary representation and its relation to topological invariants
has been studied by Frahm and the author in \cite{frahm2023twisted}, in the case of compact hyperbolic surfaces,
and by B{\'e}nard, Frahm and the author in \cite{benard2021twisted}, in the case of compact hyperbolic orbisurfaces.
Recently, Hochs and Saratchandran in \cite{hochs2023ruelle} dealt with an equivariant generalisation of Fried's conjecture 
for proper group actions on smooth manifolds, with compact quotients.

For unitary representations, the dynamical zeta functions have been studied by Bunke and Olbrich in \cite{BO} for all locally symmetric spaces of real rank 1. They proved that the zeta functions admit a meromorphic continuation to the whole complex plane and satisfy functional equations. 
Moreover, for compact, hyperbolic, odd dimensional manifolds,
and under the additional assumption that the representation is acyclic,
they proved that the Ruelle zeta function is regular at zero and equals 
the square of the analytic torsion as in {\eqref{fried}} (\cite[Theorem 4.8]{BO}).

Wotzke in \cite{Wo} extended this result for representations of $\Gamma$, which 
are not necessary unitary, but very special ones.
In particular, he considered a compact 
odd-dimensional, hyperbolic, manifold and a finite-dimensional, irreducible, representation 
$\tau\colon G\rightarrow \GL(V)$ of $G$, such that $\tau\neq\tau_{\theta}$, where 
$\tau_{\theta}=\tau\circ\theta$ and
$\theta$ denotes the Cartan involution of $G$.
Under these assumptions, he proved that the Ruelle zeta function 
admits a meromorphic continuation to the whole 
complex plane. In addition, it is regular at zero and
\begin{equation*}
 \abs{{R(0;\tau)}}=T_{\tau}(X)^2,
\end{equation*}
where the representation $\tau$ restricted to $\Gamma$.
Wotzke's method is based on the fact that if one considers the restrictions $\tau|_{K}$ and $\tau|_{\Gamma}$ of $\tau$
to $K$ and $\Gamma$, respectively, there is an isomorphism 
between the locally homogeneous vector bundle $E_{\tau}$ over $X$, 
associated with $\tau|_{K}$, and the flat vector bundle $E_{\fl}$ over $X$
associated with $\tau|_{\Gamma}$.
By \cite[Lemma 3.1]{MM}, a Hermitian 
fiber metric in $E_{\tau}$ descends to a fiber metric in $E_{\fl}$. 
Therefore, one deals with self-adjoint Laplacians and 
all tools from harmonic analysis on locally symmetric spaces are available.

\paragraph{Results.}

We consider an oriented, compact,  hyperbolic manifold $X$ of odd dimension $d$, obtained as follows. 
Let either $G=\SO^{0}(d,1)$, $K=\SO(d)$ or $G=\Spin(d,1)$, $K=\Spin(d)$. Then, $K$ is a maximal compact subgroup of $G$. Let $\widetilde{X}:=G/K$. The space $\widetilde{X}$ can be equipped with a $G$-invariant metric, which is unique up to scaling and is of constant negative curvature. If we normalize this metric such that it has constant negative curvature $-1$, then $\widetilde{X}$, equipped with this metric, is isometric to the $d$-dimensional hyperbolic space $\H^{d}$. Let $\Gamma$ be a discrete torsion-free subgroup of $G$ such that $\Gamma\backslash G$ is compact.  Then, $\Gamma$ acts by isometries on $\widetilde X$ and $X=\Gamma\backslash \widetilde X$ is a
compact, oriented, hyperbolic manifold of dimension $d$. This is a case of a locally symmetric space of non-compact type of real rank 1. This means that in the Iwasawa
decomposition $G=KAN$, $A$ is a multiplicative torus of dimension 1, i.e.,
$A\cong\R^+$. 

For a given $\gamma\in\Gamma$ we denote by $[\gamma]$ the $\Gamma$-conjugacy class
of $\gamma$. If $\gamma\neq e$, then there is a unique closed geodesic 
$c_{\gamma}$ associated with $[\gamma]$. Let $l(\gamma)$ denote the length of
$c_{\gamma}$. The conjugacy class $[\gamma]$ is called primitive if 
there exists no $k\in\N$ with $k>1$ and $\gamma_{0}\in\Gamma$ such that $\gamma=\gamma_{0}^{k}$. The prime closed geodesics correspond to the primitive conjugacy classes and are those  geodesics that trace out their image exactly once. 
Let $M$ be the centralizer of $A$ in $K$. Since $\Gamma$ is a cocompact subgroup of $G$, every element $\gamma\in \Gamma-\{e\}$ is hyperbolic.
Then, by \cite[Lemma 6.5]{Wa}, there exist a $g\in G$, a $m_{\gamma} \in M$, and an 
$a_{\gamma} \in A$, such that $ g^{-1}\gamma g=m_{\gamma}a_{\gamma}$.
The element $m_{\gamma}$ is determined up to conjugacy in $M$, and the element $a_{\gamma}$ depends only on $\gamma$.

Let $\mathfrak{g},\mathfrak{a}$ be the Lie algebras of $G$ and $A$, respectively.
Let $\Delta^{+}(\mathfrak{g},\mathfrak{a})$ be the set of positive roots of $(\mathfrak{g},\mathfrak{a})$.
Then, $\Delta^{+}(\mathfrak{g},\mathfrak{a})$ consists of a single root 
$\alpha$. Let $\mathfrak{g}_{\alpha}$ be the corresponding root space.
Let $\overline{\mathfrak{n}}$ be the negative root space of $(\mathfrak{g},\mathfrak{a})$.
Let $S^k(\Ad(m_\gamma a_\gamma)|_{\overline{\mathfrak{n}}})$ be the $k$-th
symmetric power of the adjoint map $\Ad(m_\gamma a_\gamma)$ restricted to $\overline{\mathfrak{n}}$
and $\rho$ be defined as $\rho:=\frac{1}{2}\dim(\mathfrak{g}_\alpha)\alpha$. 
We define the twisted zeta functions associated
with unitary, irreducible representations $\sigma$ of $M$ and 
finite-dimensional, complex representations $\chi$ of $\Gamma$.
The twisted Selberg zeta function $Z(s;\sigma,\chi)$ is defined for $s\in\C$ by the infinite product
\begin{equation}\label{zeta1}
Z(s;\sigma,\chi):=\prod_{\substack{[\gamma]\neq e,\\ [\gamma]\prim}} \prod_{k=0}^{\infty}\det\Big(\Id-\big(\chi(\gamma)
\otimes\sigma(m_\gamma)\otimes S^k(\Ad(m_\gamma a_\gamma)|_{\overline{\mathfrak{n}}})\big)e^{-(s+|\rho|)l(\gamma)}\Big).
\end{equation}

The twisted Ruelle zeta function $R(s;\sigma,\chi)$ is defined for $s\in\C$ by the infinite product
\begin{equation}\label{zeta2}
 R(s;\sigma,\chi):=\prod_{\substack{[\gamma]\neq{e}\\ [\gamma]\prim}}\det(\Id-(\chi(\gamma)\otimes\sigma(m_{\gamma}))e^{-sl(\gamma)}).
\end{equation}
Both $Z(s;\sigma,\chi)$ and $R(s;\sigma,\chi)$ converge absolutely and uniformly on compact subsets of some half-plane of $\C$ (\cite[Proposition 3.4 and Proposition 3.5]{Spilioti2018}).

In \cite{Spilioti2018},  it was proved that the twisted dynamical zeta functions
associated with a general, finite-dimensional, complex replresentation of
the fundamental group, as in 
in {\eqref{zeta1}} and {\eqref{zeta2}}, admit a meromorphic continuation 
to the whole complex plane. In \cite{spilioti2020functional}, the functional equations for them
are derived. Morover, it is shown a determinant formula , which relates the 
twisted Ruelle zeta function with a finite product of 
graded regularized determinants of certain,  
twisted (non-self-adjoint), elliptic differential operators
(see  \cite[Proposition 7.9, case (a)]{spilioti2020functional}).
These results extend the results of 
Bunke and Olbrich to the case of non-unitary twists.
The determinant formula is the keypoint to prove 
our main results, Theorem 1, Theorem 2 and Corollary 1, below.

In the present paper, we consider acyclic representations of $\Gamma$, which are non-unitary but closed enough to unitary ones. We provide here some details about these representations. 
For a compact, oriented, odd-dimensional, Riemannian manifold $(X,g)$ and a complex vector bundle $E\rightarrow X$, equipped with a flat connection $\nabla$, Braverman and Kappeler
in \cite{BK2} considered the odd signature operator $B=(\nabla,g)$, acting on  the space of smooth differential forms $\Lambda^{k}(X,E)$ on $X$, with values in $E$ (see Section  6.2 in \cite{BK2} and Section 6.1 in the present paper). This is a first order, elliptic differential operator,  which is in general non-self-adjoint. 
Suppose that there exists a Hermitian metric on $E$, which is preserved by $\nabla$. In such a case, we 
say that $\nabla$ is a Hermitian connection. Then, the operator is (formally) self-adjoint. Hence, if we assume  further that $\nabla$ is acyclic, i.e., $\im(\nabla|_{\Lambda^{k-1}(X,E)})=\Ker(\nabla|_{\Lambda^{k}(X,E)})$, for every $k$, the odd signature operator $B$ is bijective. 
If there is no Hermitian metric on $E$, which is preserved by $\nabla$, then the odd signature operator is no
longer self-adjoint. One can assume that $\nabla$ is acyclic, but this does not imply that the odd signature operator has a trivial kernel, since Hodge theory is no longer applicable.

However, by a continuity argument, \cite[Proposition 6.8]{BK2},  the following assumptions
\newline
\textbf{Assumption 1.}
The connection $\nabla$ is acyclic, i.e., for every $k$,
$\im(\nabla|_{\Lambda^{k-1}(X,E)})=\Ker(\nabla|_{\Lambda^{k}(X,E)})$;\\
\textbf{Assumption 2.}
The odd signature operator $B\colon \Lambda^{k}(X,E)\rightarrow \Lambda^{k}(X,E)$ is bijective;\\
are  satisfied for all flat connections in an open neighbourhood, in a suitable $C^{0}$-topology
(see Subsection 6.2), of the set of acyclic, Hermitian connections. 

Let $\Rep(\pi_{1}(X),\C^{n})$  be the set of all $n$-dimensional, complex representations of $\pi_{1}(X)$. This set has a natural structure of a complex algebraic variety (see Subsection 6.3). Each representation $\chi\in \Rep(\pi_{1}(X),\C^{n})$ gives rise to a vector bundle $E_{\chi}$ with a flat connection $\nabla_{\chi}$, whose monodromy is isomorphic to $\chi$. 
Let  $\Rep_{0}(\pi_{1}(X),\C^{n})$ be the set of all acyclic representations of $\pi_{1}(X)$, i.e., the set of representations $\chi$ such that $\nabla_{\chi}$ is acyclic. Let $\Rep^{u}(\pi_{1}(X),\C^{n})$ be the set of all unitary representations of $\pi_{1}(X)$, i.e., the set of representations such that there exists a Hermitian scalar product $(\cdot,\cdot)$ on $\C^{n}$, which is preserved by the matrices $\chi(\gamma)$, for all $\gamma\in \pi_{1}(X)$. We set
\begin{equation*}
 \Rep^{u}_{0}(\pi_{1}(X),\C^{n}):=\Rep_{0}(\pi_{1}(X),\C^{n})\cap \Rep^{u}(\pi_{1}(X),\C^{n}).
\end{equation*}
Let $B_{\chi}:=B(\nabla_{\chi},g)\colon\Lambda^{k}(X,E)\rightarrow\Lambda^{k}(X,E)$.
Let $B_{\chi}^{\ev}$ be the restriction of $B_{\chi}$ to $\Lambda^{\ev}(X,E)$.
Suppose that for some representation $\chi_{0}\in \Rep_{0}(\pi_{1}(X),\C^{n})$  the  operator $B_{\chi_{0}}$ bijective . Then, there exists an open neighbourhood (in classical topology) $V \subset \Rep(\pi_{1}(X),\C^{n})$ of the set  $\Rep^{u}_{0}(\pi_{1}(X),\C^{n})$
of acyclic, unitary representations such that, for all $\chi\in V$,
 the pair $(\nabla_{\chi},g)$ satisfies Assumptions 1 and 2. 
For such representations Braverman and Kappeler defined a Riemannian metric-invariant, 
\textit{the refined analytic torsion}, which is a non zero complex number, defined by
\begin{equation*}
T_{\chi}:={\det}_{\gr,\theta}(B_{\chi}^{\ev})e^{i\pi \rank(E_{\chi})\eta_{tr}(g)}.
\end{equation*}
Here $\theta\in (-\pi,0)$ is an Agmon angle for $B_{\chi}^{\ev}$,
${\det}_{\gr,\theta}(B_{\chi}^{\ev})$ denotes the graded determinant of 
$B_{\chi}^{\ev}$ (see Remark 7.1.2), and $\eta_{tr}(g)=\frac{1}{2}\eta(0, B_{tr}(g))$, 
where $\eta(0, B_{tr}(g))$ denotes the eta invariant of the even part of the odd signature operator corresponding to the trivial line bundle, endowed with the trivial connection (see 
Definition \ref{defref}).
One can view this definition as a special case of the definition of the refined analytic torsion as an element of the determinant line in \cite{BK3}, where the assumption of the bijectivity of the odd signature operator is removed.

On the other hand, Cappell and Miller in \cite{cm}, defined another invariant, 
\textit{the Cappell-Miller torsion} $\tau\in\det(H^{*}(X,E))\otimes \det(H^{*}(X,E))$.
In this setting, the regularized determinants of the flat Hodge Laplacians are equipped.
To define the flat Laplacians, no use of a Hermitian metric on the flat vector bundle is needed (see Section 9). These operators coincide with the square $B^{2}$ of the odd signature operator $B$. If $\chi\in V$, the element of $\det(H^{*}(X,E_{\chi}))\otimes \det(H^{*}(X,E_{\chi}))$ does not contribute in the definition (see Definition 9.2). In such a case, 
the Cappell-Miller torsion $\tau_{\chi}$ is a complex number given by
\begin{equation*}
 \tau_{\chi}:=\prod_{k=0}^{d} \det(B_{\chi,k}^{2})^{k(-1)^{k+1}}.
\end{equation*}
Then, by Proposition 9.4 (see also \cite[Subsection 5.3]{br2008cano}),
for $\chi\in V$, 
\begin{equation*}
  \tau_{\chi}=T_{\chi}^{2} e^{2\pi i(\eta(B^{\ev}_{\chi})-\rank(E_{\chi})\eta_{\tr})}, 
\end{equation*}
where  $\eta(B^{\ev}_{\chi})$ denotes the eta invariant of the even part of the odd signature operator 
$B_{\chi}$.

Let now $X$ be an oriented, compact, hyperbolic, odd-dimensional manifold $X=\Gamma\backslash\H^{d}$ as above. 
By the determinant formula in \cite{spilioti2020functional}, we obtain the following theorem.

\begin{theorem}(\Cref{detth})
Let $\chi$ be a finite-dimensional complex representation of $\Gamma$.
Let  $\Delta^{\sharp}_{\chi,k}$ be the flat Hodge Laplacian, acting on 
the space of $k$-differential forms on $X$ with values in the flat vector bundle $E_{\chi}$.
Then, the Ruelle zeta function has the representation
\begin{align*}\label{det}
R(s;\chi)=\notag&\prod_{k=0}^{d-1}\prod_{p=k}^{d-1}{{\det}_{\gr}\big(\Delta^{\sharp}_{\chi,k}+s(s+2(\vert\rho\vert-p))\big)}^{(-1)^{p}}\\
&\cdot\exp\bigg((-1)^{\frac{d-1}{2}+1}\pi(d+1)\dim (V_{\chi})\frac{\Vol(X)}{\Vol(S^{d})}s\bigg),
\end{align*}
where $\Vol(S^{d})$ denotes the volume of the $d$-dimensional Euclidean unit sphere.
Let $d_{\chi,k}:=\dim\Ker(\Delta_{\chi,k})$. Then, the singularity of the Ruelle zeta function at $s=0$ is of order
\begin{equation*}
\sum_{k=0}^{(d-1)/2}(d+1-2k)(-1)^{k}d_{\chi,k}.
\end{equation*}
\end{theorem}

Let  $V \subset \Rep(\pi_{1}(X),\C^{n})$ be an 
open neighbourhood (in classical topology) of the set  $\Rep^{u}_{0}(\pi_{1}(X),\C^{n})$
of acyclic, unitary representations such that, for all $\chi\in 
V$, $B_{\chi}$ is bijective. Then,  for $\chi\in V$, we obtain the following result.

\begin{theorem}(\Cref{main})
Let $\chi\in V$. Then, the Ruelle zeta function $R(s;\chi)$ is regular at $s=0$ and is equal 
to the complex Cappell-Miller torsion, 
\begin{equation*}
R(0;\chi)=\tau_{\chi}.
\end{equation*}
\end{theorem}

\begin{corollary}(\Cref{maincoro})
 Let $\chi\in V$. Then, the Ruelle zeta function $R(s;\chi)$ is regular at $s=0$ and is related to the refined analytic torsion  $T_{\chi}$ by
 \begin{equation*}
  R(0;\chi)=T_{\chi}^{2} e^{2\pi i(\eta(B^{\ev}_{\chi})-\rank(E_{\chi})\eta_{\tr})}.
 \end{equation*}
\end{corollary}

We mention here the work of  M\"uller in \cite{muller2021ruelle},
where the behaviour at the origin of the twisted Ruelle zeta function associated with an arbitrary non-unitary representation
is treated, and the relation to the Cappell-Miller torsion.
Also, Shen in \cite{Shen2021} studied the relation between the Cappell-Miller torsion and the value at the zero point of the Ruelle dynamical zeta function associated to a flat vector bundle, which is not necessarily unitarily flat, on a closed odd dimensional locally symmetric space of reductive type.

\paragraph{Organization of the paper}
This paper is organised  as follows. 
In Section 2, we provide the definition of the eta invariant and the
regularized determinant of a non-self-adjoint elliptic 
differential operator.
In Section 3,  we introduce the odd signature operator and
consider representations of the fundamental group, 
which are non-unitary, but closed enough, in classical topology, 
to acyclic and unitary representations. 
In Sections 4 and 5, we consider the refined analytic torsion, 
as it is introduced in \cite{BK2}. 
In Section 6, we consider the Cappell-Miller
torsion, as it is introduced in \cite{cm},
and its relation to the refined analytic torsion. 
In Section 7, we summarize basic facts about hyperbolic manifolds, 
and representations of the involved Lie groups. 
We define also the
twisted Selberg and Ruelle zeta functions.
In Section 8, we recall the definition of 
the twisted Laplace operator from \cite{M1}
and certain auxiliary operators,
first introduced in \cite{BO}.
In Section 9, we use the determinant formula 
from \cite{spilioti2020functional} to express the twisted Ruelle zeta function
in terms of products of graded regularized determinants of
twisted Hodge Laplacians on vector bundle-valued
differential forms. This formula is the keypoint to study the 
singularity of the twisted Ruelle zeta function at $s=0$.
Finally, in Section 10, we prove the main results: Theorem 2 and 
Corollary 1. We include also an Appendix, 
in Section 10, where we recall the general definition of the 
refined analytic torsion as an element of the determinant line
from \cite{BK3}.

\paragraph{Acknowledgements}
The author would like to thank  Werner M\"uller for his helpful suggestions and comments,
as well as Maxim Braverman for the insightful discussions. The author wishes to acknowledge also
the hospitality of the Department of Mathematics of the University of T\"ubingen, 
Aarhus University and University of G\"ottingen. This work was partially funded by
a research grant from the Villum Foundation (Grant No. 00025373) and 
the Research Training Group RTG 2491-Fourier Analysis and Spectral Theory.

\section{Preliminaries on the Eta invariant of a non self-adjoint operator and graded regularized determinants}

\subsection{Eta invariant of an elliptic operator}

Let $E\rightarrow X $ be a complex vector bundle over a smooth, compact, Riemannian manifold X of dimension d. Let $D\colon C^{\infty}(X, E)\rightarrow C^{\infty}(X, E)$ be an elliptic differential operator of order $m\geq 1$. Let $\sigma(D)$ be its principal symbol.

\begin{defi}
Let  $R_{\theta}:=\{\rho e^ {i\theta}: \rho\in[0,\infty]\}$. 
The angle $\theta\in[0,2\pi)$ is a principal angle for $D$ if 
\begin{equation*}
 \spec(\sigma_{D}(x,\xi))\cap R_{\theta}=\emptyset,\quad \forall x\in X,\forall\xi\in T_{x}^{*}X,\xi\neq 0.
\end{equation*}
\end{defi}
\begin{defi}
Let $I\subset [0,2\pi)$.
Let $L_{I}$ be a solid angle defined by
\begin{equation*}
 L_{I}:=\{\rho e^ {i\theta}: \rho\in(0,\infty),\theta\in I\}.
\end{equation*}
The angle $\theta$ is an Agmon angle for $D$, if it is a principal angle for $D$ and there exists an $\varepsilon>0$ such that 
\begin{equation*}
 \spec(D)\cap L_{[\theta-\varepsilon,\theta+\varepsilon]}=\emptyset.
\end{equation*}
\end{defi}
We define here the eta function associated with non-self-adjoint operators  $D$ with elliptic, self-adjoint principal symbol (see \cite{Gilkeysoviet}).
By \cite[\S I.6]{Mk}, the space $L^{2}(X,E)$ of square integrable sections 
of $E$ is the closure of the algebraic direct sum of finite dimensional $D$-invariant subspaces
\begin{equation*}
L^{2}(X,E)=\overline{\bigoplus\limits_{k\geq 1} \Lambda_{k}},
\end{equation*}
such that the restriction of $D$ to $\Lambda_{k}$
has a unique eigenvalue $\lambda_{k}$
and $\lim_{k\rightarrow \infty}\vert \lambda_{k}\vert=\infty$.
In general, the above sum is not a sum of mutually orthogonal subspaces. 
The space $\Lambda_{k}$ is called the space of root vectors 
of $D$ with eigenvalue 
$\lambda_{k}$.  
The algebraic multiplicity $m_{k}$ of the the eigenvalue $\lambda_{k}$ is defined as the  the dimension of the space $\Lambda_{k}$.
Since the principal symbol of $D$ is self-adjoint,  the angles 
$\pm \pi /2$ are principal angles for $D$. 
Hence, $D$ also possesses an  Agmon angle (see \cite[Section 3.10]{BK2}). Let $\theta$ be an Agmon angle for $D$. Denote by $\log_{\theta}\lambda_{k}$ the branch of the logarithm in $\C\setminus R_{\theta}$ with  $\theta<\im(\log_{\theta}\lambda_{k})<\theta+2\pi$. Let $(\lambda_{k})_{\theta}:=e^{\log_{\theta}\lambda_{k}}$.
\begin{defi}
For Re$(s)\gg0$, we define the eta function $\eta_{\theta}(s,D)$ of $D$ by 
\begin{equation*}
 \eta_{\theta}(s,D):=\sum_{\text{\Re}(\lambda_{k})>0} m_{k}(\lambda_{k})_{\theta}^{-s}-\sum_{\text{\Re}(\lambda_{k})<0} m_{k}(-\lambda_{k})_{\theta}^{-s}.
\end{equation*}
\end{defi}
Note that since the angles $\pm \pi/2$ are principal angles for $D$, there are at most finitely many eigenvalues of $D$ on or near the imaginary axis.
Hence, the eigenvalues  of $D$ that are purely imaginary do not
contribute to the definition of the eta function. 
It has been shown by Grubb and Seeley (\cite[Theorem 2.7]{GS}) that $\eta_{\theta}(s,D)$ has a meromorphic continuation to the whole complex plane $\C$ with isolated simple poles and that is regular at $s=0$. Moreover, the number $\eta_{\theta}(0,D)$ is independent of the Agmon angle $\theta$. Hence, we write $\eta(0,D)$ instead of $\eta_{\theta}(0,D)$.
\begin{defi}
 Let $m_{+}$, respectively $m_{-}$, denote the the number of the eigenvalues of $D$
on the positive, respectively negative, part of the imaginary axis.
We define the eta invariant $\eta(D)$ of the operator $D$ by 
\begin{equation*}
\eta(D)=\frac{\eta(0,D)+m_{+}-m_{-}}{2}.
\end{equation*}
\end{defi}
By \cite[(3.25)]{BK2}, $\eta(D)$ is independent of the Agmon angle $\theta$.

\subsection{Regularized determinant of an elliptic operator}
We recall here the definition  of the regularized determinant of an elliptic operator.
For more details we refer the reader to \cite[Subsection 3.5 and Definition 3.6]{BK2}.
Let $X,E$ and $D$ be as in the previous subsection.
We assume, in addition, that $D$ has a self-adjoint principal symbol and 
it is invertible. Let $\theta$ be an Agmon angle for $D$.
We define the zeta function $\zeta_{\theta}(s,D)$ by
\begin{equation*}
\zeta_{\theta}(s,D):=\Tr(D_{\theta}^{-s}),\quad \text{Re}(s)>\frac{\dim(X)}{m},
\end{equation*}
where $D_{\theta}^{-s}$ is a pseudo-differential operator with continuous kernel
(see \cite[(3.15)]{BK2}).  The zeta function $\zeta_{\theta}(s,D)$ admits a meromorphic continuation to 
$\C$ and it is regular at zero (\cite{seeley1967complex}).
We define the regularized determinant of $D$ by
\begin{equation*}
{\det}_{\theta}:=\exp\bigg( -\frac{d}{ds}\bigg\vert_{s=0} \zeta_{\theta}(s,D) \bigg).
\end{equation*}
We denote by $L{\det}_{\theta}$ the particular value of the logarithm of the determinant, such that
\begin{equation*}
 L{\det}_{\theta}(D)=-\zeta_{\theta}'(0,D).
\end{equation*}
By \cite[Subsection 3.10]{BK2}, the regularized determinant ${\det}_{\theta}(D)$
is independent of the Agmon angle $\theta$.

\subsection{Graded regularized determinant of an elliptic operator}
We recall here the notion of the graded regularized determinant of an elliptic differential operator.
Let $E=E_{+}\oplus E_{-}$ be a $\Z_{2}$-graded  vector bundle over a compact,
Riemannian manifold $X$. Let $D\colon C^{\infty}(X,E)\rightarrow C^{\infty}(X,E)$ be an elliptic differential operator,
which is bounded from below. We assume that $D$ preserves the grading, i.e.,
we assume that with respect to the decomposition
\begin{equation*}
C^{\infty}(X,E)=C^{\infty}(X,E^{+})\oplus C^{\infty}(X,E^{-}),
\end{equation*}
$D$ takes the form,
\begin{equation*}
   D=
  \left( {\begin{array}{cc}
   D{+} & 0 \\
   0 &    D_{-} \\
  \end{array} } \right).
\end{equation*}
Then, the graded determinant $\det_{\gr}(D)$ of $D$ is defined by
\begin{equation*}
{\det}_{\gr}(D)=\frac{\det(D_{+})}{\det(D_{-})},
\end{equation*}
where $\det(D_{+})$ and $ \det(D_{-})$ denote the 
regularized determinants of the operators $D_{+}$ and 
$D_{-}$, respectively.

\section{Spaces of representations of the fundamental group}
\subsection{Odd signature operator}

Following the idea of Braverman and Kappeler in \cite{BK2}, we consider representations of $\Gamma$, which are non-unitary, but belong to a neighbourhood of the set of acyclic, unitary representations in a suitable topology, such that the odd signature operator $B$
(see Definition 3.1.1 below) is bijective. 
For these representations, Braverman and Kappeler defined the refined analytic torsion for a compact, oriented, 
odd-dimensional manifold $X$. Having applications in mind,  instead of considering the regularized determinant of the (flat) Laplacian  they consider the graded determinant of the even part of 
its square root, the odd signature operator $B$.
We recall here some definitions from \cite{BK2}.

Let $X$ be a compact, oriented, Riemannian manifold of odd dimension $d=2r-1$. Denote by $g$ the Riemannian metric on $X$. Let $E\rightarrow X$ be a complex vector bundle over $X$, endowed with a flat connection $\nabla$. Let $\Lambda^{k}(X,E)$ be the space of $k$-differential forms on $X$ with values in $E$. 
We denote by $\nabla$ the induced differential 
$\nabla\colon \Lambda^{k}(X,E)\rightarrow \Lambda^{k+1}(X,E)$.
Let $\Gamma\colon \Lambda^{k}(X,E) \rightarrow\Lambda^{d-k}(X,E)$ be the chirality operator defined by the formula
\begin{equation*}
 \Gamma\omega:=i^{(d+1)/2}(-1)^{k(k+1)/2}\ast' \omega,
\end{equation*}
where $\ast'$ denotes the operator acting on sections of $\Lambda^{k}T^{\ast}X\otimes E$ as $\ast\otimes \Id$,
and $\ast$ is the usual Hodge $\ast$-operator.
\begin{defi}\label{oddsign}
We define the odd signature operator $B=B(\nabla, g)$ acting on $\Lambda^{*}(X,E)$ by 
\begin{equation}\label{odd}
 B:=\Gamma \nabla+\nabla \Gamma.
\end{equation} 
\end{defi}
We set
\begin{equation*}
 \Lambda^{\ev}(X,E):=\bigoplus_{p=0}^{r-1} \Lambda^{2p}(X,E).
\end{equation*}
Let $B^{\ev}$ be the even part of the odd signature operator $B$, defined by 
$B^{\ev}:=B\colon \Lambda^{\ev}(X,E)\rightarrow\Lambda^{\ev}(X,E)$.

\subsection{Spaces of connections}

We consider the following two assumptions.\\
\newline
\textbf{Assumption 1.}
The connection $\nabla$ is acyclic, i.e.,
the twisted de Rham complex
\begin{equation*}
0\rightarrow \Lambda^{0}(X,E)
\overset{\nabla}{\longrightarrow} \Lambda^{1}(X,E)
\overset{\nabla}{\longrightarrow}\cdots 
\overset{\nabla}{\longrightarrow}\Lambda^{d}(X,E)
\rightarrow 0
\end{equation*}
is acyclic,
\begin{equation*}
\im(\nabla|_{\Lambda^{k-1}(X,E)})=\Ker(\nabla|_{\Lambda^{k}(X,E)}), 
\end{equation*}
for every $k=1,\ldots,d$.
\newline

\textbf{Assumption 2.}
The odd signature operator $B\colon \Lambda^{k}(X,E)\rightarrow \Lambda^{k}(X,E)$ is bijective.\\

Suppose that there exists a Hermitian metric $h$ on $E$, which is preserved by $\nabla$. 
In such a case, we call $\nabla$ a Hermitian connection.
For such connections, one can easily see that Assumption 1 implies Assumption 2 and vice versa (see \cite[Subsection 6.6]{BK2}).
Hence, all acyclic, Hermitian connections satisfy Assumptions 1 and 2.

Following \cite{BK2}, we denote by $\Lambda^{1}(X,\End(E))$ be the space of $1$-forms with values in $\End(E)$. We define the sup-norm on $\Lambda^{1}(X,\End(E))$  by 
\begin{equation*}
 \lVert\omega\rVert_{\sup}:= \max_{x\in X}\lvert \omega(x)\rvert,
\end{equation*}
where the norm $\lvert\cdot \rvert$ is induced by a Hermitian metric on $E$ and a Riemannian metric on $X$. The topology defined by this norm is called the $C^{0}$-topology and it is independent of the metrics.
We identify the space of connections on $E$ 
with $\Lambda^{1}(X,\End(E))$,
by choosing a connection $\nabla_{0}$ and
associating to a connection $\nabla$ the $1$-form 
$\nabla-\nabla_{0}\in\Lambda^{1}(X,\End(E))$.
Hence, by this identification, 
the $C^{0}$-topology on $\Lambda^{1}(X,\End(E))$
provides a topology on 
the space of connections on $E$, which 
is independent of the choice of $\nabla_{0}$. 
This topology is called the $C^{0}$-topology on the space 
of connections.

Let $\Flat(E)$ be the set of flat connections on $E$ and $\Flat'(E,g)\subset \Flat(E)$ be the set of  
flat connections on $E$, 
satisfying Assumptions 1 and 2.
The topology induced on these sets by the $C^{0}$-topology on 
the space of connections on $E$
is also called the $C^{0}$-topology.
The following proposition is proved in \cite{BK2}.
\begin{prop}
 $\Flat'(E,g)$ is an $C^{0}$-open subset of $\Flat(E)$, which contains 
 all acyclic Hermitian connections on $E$.
\end{prop}
\begin{proof}
 See \cite[Proposition 6.8]{BK2}.
\end{proof}

\subsection{Spaces of representations}\label{Vnei}

Let \(\Rep(\pi_{1}(X),\C^{n})\) be the space of all $n$-dimensional, complex representations of $\pi_{1}(X)$. 
This space has a natural structure of a complex algebraic variety.
Let $\{\gamma_{1},\ldots,\gamma_{L}\}$ be a finite set of generators of $\pi_{1}(X)$. 
The elements $\gamma_{i}$ satisfy finitely many relations. Then, a representation $\chi\in \Rep(\pi_{1}(X),\C^{n})$
is given by $2L$ invertible $n\times n$-matrices
$\chi(\gamma_{1}),\ldots,\chi(\gamma_{L}),\chi(\gamma_{1}^{-1}),\ldots,\chi(\gamma_{L}^{-1})$, which satisfy
finitely many polynomial equations. 
Hence, we  view $\Rep(\pi_{1}(X),\C^{n})$
as an algebraic subset of $\Mat_{n\times n}(\C)^{2L}$ with the induced topology.

For a representation $\chi\in \Rep(\pi_{1}(X),\C^{n})$, we consider the associated flat vector bundle
$E_{\chi}\rightarrow X$.  Let $\nabla_{\chi}$
be the flat connection on $E_{\chi}$. 
Then, the monodromy of $\nabla_{\chi}$ is isomorphic to $\chi$.
We denote also by $\nabla_{\chi}$ the induced differential 
$\nabla_{\chi}\colon\Lambda^{*}(X,E_{\chi})\rightarrow\Lambda^{*+1}(X,E_{\chi})$.
We denote by $B_{\chi}$ the odd signature operator $B_{\chi}=B(\nabla_{\chi},g)$ acting on $\Lambda^{*}(X,E_{\chi})$.
Let $B_{\chi}^{\ev}$ be the restriction of $B_{\chi}$ to $\Lambda^{\ev}(X,E_{\chi})$.

Let $\Rep_{0}(\pi_{1}(X),\C^{n})$ be the set of all acyclic representations of $\pi_{1}(X)$, i.e., 
the set of all representations $\chi$ such that $\nabla_{\chi}$ is acyclic. 
Let $\Rep^{u}(\pi_{1}(X),\C^{n})$ be the set of all unitary representations of $\pi_{1}(X)$, i.e., 
the set of all representations $\chi$ such that there exists a
Hermitian scalar product $(\cdot,\cdot)$ on $\C^{n}$, 
which is preserved by the matrices $\chi(\gamma)$, for all $\gamma\in \pi_{1}(X)$. 
$\Rep^{u}(\pi_{1}(X),\C^{n})$ can be viewed as the real locus of the complex algebraic 
variety $\Rep(\pi_{1}(X),\C^{n})$. We set
\begin{equation*}
 \Rep^{u}_{0}(\pi_{1}(X),\C^{n}):=\Rep_{0}(\pi_{1}(X),\C^{n})\cap \Rep^{u}(\pi_{1}(X),\C^{n}).
\end{equation*}
Suppose that for some representation $\chi_{0}\in \Rep_{0}(\pi_{1}(X),\C^{n})$  the  operator $B_{\chi_{0}}$ bijective . Then, there exists an open neighbourhood (in classical topology) $V \subset \Rep(\pi_{1}(X),\C^{n})$ of the set  $\Rep^{u}_{0}(\pi_{1}(X),\C^{n})$
of acyclic, unitary representations such that, for all $\chi\in 
V$, the pair $(\nabla_{\chi},g)$ satisfies Assumptions 1 and 2 of Subsection 6.2
(see \cite[Subsection 13.7]{BK2}).

\section{Refined analytic torsion}
\subsection{Definition}
Let $\nabla$ be acyclic. 
Let $\cal{M}(\nabla)$ be the set of all Riemannian metrics $g$ on $X$, 
such that $B^{\ev}=B^{\ev}(\nabla,g)$ is bijective.
Then, by \cite[Proposition 6.8]{BK2}, $\cal{M}(\nabla)\neq \emptyset$, for all flat connections in an open neighbourhood, in $C^{0}$-topology, of the set of acyclic Hermitian connections.
Let $V$ be an open neighbourhood of  $\Rep^{u}_{0}(\pi_{1}(X),\C^{n})$ as in Subsection 3.3.

\begin{defi}\label{defref}
 Let $\chi\in V$. We define the refined analytic torsion $T_{\chi}=T(\nabla_{\chi})$ by
\begin{equation}\label{ref}
T_{\chi}:={\det}_{\gr,\theta}(B_{\chi}^{\ev})e^{i\pi \rank(E_{\chi})\eta_{tr}}\in \C\backslash\{0\},
\end{equation}
where $\eta_{\tr}=\eta_{tr}(g)=\frac{1}{2}\eta (0, B_{tr}(g))$, and $\eta (0, B_{tr}(g))$ denotes the eta invariant of the even part of the odd signature operator corresponding to the trivial line bundle, endowed with the trivial connection.
The refined analytic torsion is independent of the Agmon angle $\theta\in(-\pi,0)$ for the operator
 $B_{\chi}^{\ev}=B_{\chi}^{\ev}(g)$ and the Riemannian metric $g\in \cal{M}(\nabla)$.
\end{defi}

\begin{rmrk}
The graded determinant of the operator $B_{\chi}^{\ev}$ is 
slightly different  than the usual expression of the graded determinant
as in Subsection 2.2.3. We recall here briefly the definition of 
${\det}_{\gr}(B_{\chi}^{\ev})$. For more details, we refer
the reader to \cite[Subsection 2.2 and 6.11]{BK2}.
We set
\begin{align*}
&\Lambda^{k}_{+}(X,E):=\Ker(\nabla\Gamma)\cap\Lambda^{k}(X,E)\\
&\Lambda^{k}_{-}(X,E):=\Ker(\Gamma\nabla)\cap\Lambda^{k}(X,E).
\end{align*}
Assumption 2 in Subsection 3.2 implies that $\Lambda^{k}(X,E)=\Lambda^{k}_{+}(X,E)\oplus\Lambda^{k}_{-}(X,E)$
(see \cite[Subsection 6.9]{BK2}). 
We set 
\begin{equation*}
\Lambda^{\ev}_{\pm}(X,E)=\bigoplus_{p=0}^{r-1} \Lambda^{2p}_{\pm}(X,E).
\end{equation*}
We denote by $B_{\chi,\pm}^{\ev}$ the restriction of
$B_{\chi}^{\ev}$ to $\Lambda^{\ev}_{\pm}(X,E)$. Then, $B_{\chi}^{\ev}$ 
leaves the subspaces $\Lambda^{\ev}_{\pm}(X,E)$ invariant. 
By Assumption 2, the operators $B_{\chi,\pm}^{\ev}$ are bijective. 
Hence, we define the graded determinant ${\det}_{\gr}(B_{\chi}^{\ev})\in\C
\backslash\{0\}$ of $B_{\chi}^{\ev}$ by
\begin{equation*}
{\det}_{\gr,\theta}(B_{\chi}^{\ev}):=\frac{\det_{\theta}(B_{\chi,+}^{\ev})}
{\det_{\theta}(-B_{\chi,-}^{\ev})},
\end{equation*}
where $\theta$ is an Agmon angle for $B_{\chi}^{\ev}$
(and an Agmon angle for $B_{\chi,\pm}^{\ev}$ as well).
\end{rmrk}
\begin{rmrk}
By \cite[Corollary 13.11.(2)]{BK2}, the refined analytic torsion $T_{\chi}$ is an holomorphic function on the set $V\backslash \Sigma$, where $\Sigma$ denotes the set of singular points of the complex algebraic variety $\Rep(\pi_{1}(X),\C^{n})$.
\end{rmrk}

\subsection{Independence of the refined analytic torsion of the Riemannian metric and the Agmon angle}

The independence of $T_{\chi}$ of the Agmon angle $\theta$ and the Riemannian metric $g$ is proved in \cite[Theorem 9.3]{BK2}. 
We set
\begin{equation}\label{xi}
 \xi=\xi(\chi,g,\theta):=\frac{1}{2}\sum_{k=0}^{d}(-1)^{k+1}k L{\det}_{2\theta}(B_{\chi}^{2}|_{\Lambda^{k}(X,E_{\chi})}).
\end{equation}
The number $\xi$ is defined in \cite[(7.79)]{BK2} (see also (8.102) in the same paper). 
Let  $\eta=\eta(\chi,g)=\eta(B^{\ev}_{\chi})$ denote the eta invariant of the even part of the odd signature operator 
$B_{\chi}$.
By \cite[Theorem 7.2]{BK2}, for a suitable choice of an Agmon angle for $B_{\chi}$, the following equality holds
\begin{equation}\label{detcrucial}
 {\det}_{\gr,\theta}(B_{\chi}^{\ev})=e^{\xi(\chi,g,\theta)}e^{-i\pi\eta(\chi,g)}.
\end{equation}
The graded determinant ${\det}_{\gr}(B_{\chi}^{\ev})$ above 
does depend on the Riemannian metric $g$. This is the reason why the additional exponential $e^{i\pi \rank(E_{\chi})\eta_{tr}(g)}$ is considered in the definition of the refined analytic torsion in \cite{BK2}, to remove the metric anomaly. By results in 
\cite[p. 52]{Gilkeysoviet}, modulo $\Z$, the difference 
\begin{equation*}
 \eta(\chi,g)-\rank(E_{\chi})\eta_{tr}(g)
\end{equation*}
is independent of $g$. In addition, by \cite[Proposition 9.7]{BK2}, 
for $g_{1},g_{2}\in\mathcal{M}(\nabla)$ and suitable choices of  Agmon angles $\theta_{0},\theta_{1}$,
\begin{equation*}
 \xi(\chi,g_{1},\theta_{1})=\xi(\chi,g_{2},\theta_{2})\mod\pi i.
\end{equation*}
Hence, for different choices of Agmon angles and Riemannian metrics, the corresponding expressions in 
\eqref{ref} coincide up to a sign. To see that the two expressions coincide see \cite[p. 238]{BK2}.

\begin{rmrk}
\begin{equation*}
\rho_{\chi}:=\eta(\chi,g)-\rank(E_{\chi})\eta_{tr}(g),
\end{equation*}
is called the $rho$ invariant $\rho_{\chi}$ of the operator $B_{\chi}^{\ev}$.
\end{rmrk}

\section{Alternative definition of the refined analytic torsion}

In  \cite[Section 11]{BK2}, an alternative definition of the refined analytic torsion is given,
which involves the Hirzebruch  $L$-polynomial and the signature theorem by Atiyah, Patodi and Singer. We recall this definition here.

Let $N$ be a smooth, oriented, compact, even-dimensional  manifold such that $\partial{N}=X$.  
Recall that the signature $\sign(N)$ is an integer defined in pure cohomological terms
(see \cite[p. 65]{ATP} and \cite[p. 407]{atiyah1975spectral}), 
as the signature of the Hermitian form, induced by 
the cup product,  in the middle ($\dim N/2$)-cohomology,
and it is metric independent.
We denote by $L(p)$ the Hirzebruch $L$-polynomial 
in the Prontrjagin forms of a Riemannian metric of a manifold (\cite[p. 228, 232]{LM}).
The signature theorem (\cite[Theorem 4.14]{ATP} , \cite[Theorem 2.2]{atiyah1975spectral}) states
\begin{equation}\label{sign}
\sign (N)=\int_{N}L(p)-\eta(B_{trivial}),
\end{equation}
where   $L_{N}(p):=L(p)$ denotes the Hirzebruch $L$-polynomial 
in the Prontrjagin forms of a Riemannian metric on $N$ which is a product near $X$.
By the metric independence of $\sign (N)$,
\eqref{sign} and Proposition 9.5 in \cite{BK2}, 
we have that modulo $\Z$,
\begin{equation*}
\eta-\rank (E) \int_{N}L(p)
\end{equation*}
is independent of the Riemannian metric on $X$.

In general, there might be no smooth, oriented manifold $N$ such that
$\partial N=X$. However, since $\dim X$ is odd, there exists an oriented manifold 
$N$, whose oriented boundary is the disjoint union of two copies of $X$, with the same orientation (\cite{wall1960determination}, \cite[Th. IV.6.5]{rudyak1998thom}).
Then, as discussed above, modulo $\Z$,
\begin{equation*}
\eta-\frac{\rank (E)}{2} \int_{N}L(p).
\end{equation*}
is independent of the Riemannian metric on $X$.

We assume that the flat connection $\nabla$ on $E$ belongs to an open 
neighbourhood, in $C^{0}$-topology, of the set of acyclic Hermitian connections 
(see Subsection 3.2).
\begin{defi}
Let $\theta\in (-\pi, 0)$ be the Agmon angle of $B^{\ev}$.  Choose  a smooth, compact, oriented manifold $N$,  whose oriented boundary is diffeomorphic to two disjoint copies of $M$. The refined analytic torsion $T'(\nabla)$ is defined by
\begin{equation}\label{altrefdef}
T'(\nabla)= T'(X, E, \nabla, N):= {\det}_{\gr,\theta}(B^{\ev})e^{i \pi\frac{\rank (E)}{2} \int_{N}L(p)},
\end{equation}
where where  $L_{N}(p):=L(p)$ denotes the Hirzebruch $L$-polynomial 
in the Prontrjagin forms of a Riemannian metric on $N$ which is a product near $\partial N$.
\end{defi}

By the discussion above and the same arguments as 
in the proof of
 \cite[Theorem 9.3]{BK2}, $T'(\nabla)$
is independent the Agmon angle $\theta$ and
of the Riemannian metric on $X$. 
However, $T'(\nabla)$ does depend on the choice of $N$. 
Since for different choices of $N$, $\int_{N}L(p)$  differs by an integer,
$T'(\nabla)$ is independent of the choice of $N$ up to a 
multiplication by $i^{k\rank (E)}$, $k\in \Z$.
Hence, if $\rank (E)$ is even, $T'(\nabla)$ is well defined up to a sign, and if $\rank(E)$
is divisible by $4$, it is a well defined complex number.
In addition, by \eqref{altrefdef} and \cite[Theorem 7.2]{BK2}, we have
\begin{equation}\label{altref}
T'(\nabla)=  e^{\xi}e^{-i\pi \eta}e^{i \pi\frac{\rank (E)}{2} \int_{N}L(p)},
\end{equation}
where $\xi$ is as in \eqref{xi}.

\section{Cappell-Miller complex torsion}

In \cite{cm}, Cappell and Miller introduced the Cappell-Miller torsion $\tau$,
an invariant, which is an element of $\det(H^{*}(X,E))\otimes \det(H^{*}(X,E))$.
To define this invariant, they used the flat Laplacian on vector bundle-valued differential forms. 
This operator is non-self-adjoint and in terms of the chirality operator in the previous sections is the square $B^{2}$ of the odd signature operator $B$. Hence, roughly speaking, Braverman and Kappeler  used the odd signature operator to define the refined analytic torsion and Cappell and Miller its square.
By \cite[eq. (5.1)]{br2008cano},
for a finite-dimensional complex $(C^{*},\partial)$, 
\begin{equation*}
 \tau=\tau_{\Gamma}:=\rho_{\Gamma}\otimes\rho_{\Gamma}\in\det(H^{*}(\partial))
 \otimes\det(H^{*}(\partial)),
\end{equation*}
where $\rho_{\Gamma}\in\det(H^{*}(\partial))$ is 
as in Definition 11.1.1 (Appendix). If the complex $(C^{*},\partial)$ is acyclic then $\tau_{\Gamma}$ is a complex number and by \cite[eq. (5.3)]{br2008cano},
\begin{equation}\label{cm}
  \tau_{\Gamma}:=\prod_{k=0}^{d} \det(B^{2}|_{C^{k}})^{k(-1)^{k+1}}
\end{equation}
\begin{rmrk}
As explained in  \cite[Remark 5.2]{br2008cano}, the element $\tau$ 
in \cite{cm} is constructed slightly different. In particular, by the construction of the element $\rho_{\Gamma}$
and the construction in \cite[Section 6]{cm},
the two elements coincide up to a sign. One can see that the signs agree by \eqref{cm} and equation (6.18) in \cite{cm}.
\end{rmrk}

In the manifold setting, we consider the Cappell-Miller torsion $\tau_{\Gamma_{[0,\lambda]}}\in\det(H^{*}_{[0,\lambda]}(X,E))\otimes\det(H^{*}_{[0,\lambda]}(X,E)) $ of the finite-dimensional complex  $\Lambda^{k}_{[0,\lambda]}(X,E)$
(see  Appendix, Subsection 11.2).

\begin{defi}
Let $\theta\in (0,2\pi)$ be an Agmon angle for the  operator $B^{2}_{(\lambda,\infty)}$. Then, the Cappell-Miller torsion  $\tau_{\nabla}$ is defined by
\begin{equation*}
 \tau_{\nabla}:= \tau_{\Gamma_{[0,\lambda]}}\prod_{k=0}^{d}{\det}_{\theta}(B^{2}|_{\Lambda^{k}_{(\lambda,\infty)}(X,E)})^{k(-1)^{k+1}}\in \det(H^{*}(X,E))\otimes\det(H^{*}(X,E)). 
\end{equation*}
\end{defi}
By \cite[Theorem 8.3]{cm}, the Cappell-Miller torsion $\tau_{\nabla}$ is independent of the 
choice of $\lambda>0$ and the choice of the Riemannian metric on $X$.
Moreover, since the principal symbol of $B^{2}$ is self-adjoint, 
by \cite[Section 3.10]{BK2}, the regularized determinant 
${\det}_{\theta}(B^{2}|_{\Lambda^{k}_{(\lambda,\infty)}(X,E)})$
is independent of the choice of the Agmon angle $\theta$.
\begin{rmrk}
 In the presence of the chirality operator $\Gamma$, the dual differential $\partial^{*, \sharp}$, introduced in \cite{cm} is  given by $\partial^{*, \sharp}=\Gamma \partial \Gamma$. Then, 
 $(\partial^{*, \sharp})^{2}=0$ and one obtains a complex 
 $(\Lambda^{*}(X,E),\partial^{*, \sharp})$ of degree $-1$.
 In addition, the flat Laplacian $\Delta^{\sharp}$ acting on $\Lambda^{*}(X,E)$ is given by
\begin{equation*}
 \Delta^{\sharp}:=\partial \partial^{*,\sharp}+\partial^{*,\sharp}\partial.
\end{equation*}
This operator is the same as the square $B^{2}$ of the odd signature operator, introduced in Section 3.
Then, the Cappell-Miller torsion $\tau$ is defined in terms of the bicomplex $(\Lambda^{*}(X,E),\partial, \partial^{*, \sharp})$.
In particular, it is given as a combination of the $\mathbf{torsion}(\Lambda^{k}_{[0,\lambda]}(X,E),\partial,\partial^{*, \sharp})\in \det(H^{*}(X,E))\otimes\det(H^{*}(X,E))$
and the square of the $\mathbf{Ray-Singer}$ term,
which corresponds to the terms $\tau_{\Gamma_{[0,\lambda]}}$ and $\prod_{k=0}^{d} {\det}{\theta}(B^{2}|_{\Lambda^{k}_{(\lambda,\infty)}(X,E)})^{k(-1)^{k+1}}$, respectively, in Definition 9.2.
\end{rmrk}

We turn now to the case, where the flat connection $\nabla$ 
on the flat vector bundle belongs to the neighbourhood of the set 
of all acyclic Hermitian connections. Let $\chi\in V$, where $V$ is as in Subsection 6.3.
Then, $B_{\chi}$ is bijective, the cohomologies $H^{*}(X,E_{\chi})$ vanish and the element $\tau_{\Gamma_{[0,\lambda]}}$ is equal to $1$.
For $\chi\in V$, we define the Cappell-Miller torsion $\tau_{\chi}:=\tau(\nabla_{\chi})$ by
\begin{equation}\label{cmchi}
 \tau_{\chi}:=\prod_{k=0}^{d} \det(B_{\chi,k}^{2})^{k(-1)^{k+1}},
\end{equation}
where $B_{\chi,k}^{2}$ denotes the operator $B_{\chi}^{2}$
acting on $\Lambda^{k}(X,E)$.
In this case $\tau_{\chi}$ is a complex number.

\begin{prop}
 Let $\chi\in V$. Then, the Cappell-Miller torsion $\tau_{\chi}$ and the refined analytic torsion $T_{\chi}$ are related by
 \begin{equation}\label{comparison}
  \tau_{\chi}=T_{\chi}^{2} e^{2\pi i(\eta(B^{\ev}_{\chi})-\rank(E_{\chi})\eta_{\tr})}.
 \end{equation}
\end{prop}
\begin{proof}
\eqref{comparison} follows from \eqref{ref}, \eqref{xi}, \eqref{detcrucial} and \eqref{cmchi}.
\end{proof}

\section{Twisted Selberg and Ruelle zeta functions}

In this paper, we are dealing with odd-dimensional, compact,  hyperbolic manifolds,
obtained as follows. Let  $d=2n+1$, $n\in\N_{>0}$.
We consider the universal coverings $G=\Spin(d,1)$ of $\mathrm{SO}^0(d,1)$ and $K=\Spin(d)$ of $\mathrm{SO}(d)$, respectively.
We set  $\widetilde{X}:=G/K$. 
Let $\mathfrak{g},\mathfrak{k}$ be the Lie algebras of $G$ and $K$, respectively. 
Let $\Theta$ be the Cartan involution of $G$ and 
$\theta$ the differential of $\Theta$ at $e$, where $e$ is the identity element of $G$.
Let $ \mathfrak{g}=\mathfrak{k}\oplus\mathfrak{p}$ be the Cartan decomposition of $\mathfrak{g}$
with respect to $\theta$.
There exists a canonical isomorphism $T_{eK}\cong \mathfrak{p}$.
Let $B(X,Y)$ be the Killing form on $\mathfrak{g}\times \mathfrak{g}$ defined by $ B(X,Y)=\Tr (\ad(X)\circ \ad(Y))$. It is a symmetric bilinear form. 
We set
\begin{equation}\label{ip}
 \langle Y_1, Y_2\rangle:=\frac{1}{2(d-1)}B(Y_1,Y_2),\quad Y_1,Y_2 \in \mathfrak{g}.
\end{equation}
The restriction of $\langle\cdot,\cdot\rangle$ to $\mathfrak{p}$ defines
an inner product on $\mathfrak{p}$ and therefore induces a $G$-invariant Riemannian metric on $\widetilde{X}$, which has constant curvature $-1$. Then, $\widetilde{X}$ equipped with this metric is isometric to $\H^{d}$.
Let $\Gamma_{1}$ be a torsion-free, cocompact, discrete subgroup of $\mathrm{SO}^0(d,1)$.
We assume that $\Gamma_{1}$ can be lifted to a subgroup $\Gamma$ of $G$. Then, $X:=\Gamma\backslash \widetilde{X}$ is a compact, hyperbolic manifold of odd dimension $d$.

\subsection{Representation theory of Lie groups}

Let 
\begin{equation*}
 G=KAN
\end{equation*}
be the standard Iwasawa decomposition of $G$
and let $M$ be the centralizer of $A$ in $K$. 
Then, $M=\Spin(d-1)$. 
Let $\mathfrak{a}$ and $\mathfrak{m}$
be the Lie algebras of $A$ and $M$,
respectively.
Let $\mathfrak{b}$ be a Cartan subalgebra of $\mathfrak{m}$. Let $\mathfrak{h}$ be a Cartan subalgebra of $\mathfrak{g}$. We consider the complexifications 
$\mathfrak{g}_{\C}:=\mathfrak{g}\oplus i\mathfrak{g}$, $\mathfrak{h}_{\C}:=\mathfrak{h}\oplus i\mathfrak{h}$
and $\mathfrak{m}_{\C}:=\mathfrak{m}\oplus i\mathfrak{m}$.

Let $\Delta^{+}(\mathfrak{g},\mathfrak{a})$ be the set of positive roots of $(\mathfrak{g},\mathfrak{a})$ and $\mathfrak{g}_{\alpha}$ the corresponding root spaces.
In the present case $\Delta^{+}(\mathfrak{g},\mathfrak{a})$ consists of a single root 
$\alpha$. Let  $M'=\Norm_{K}(A)$ be the normalizer of $A$ in $K$.
Let $H\in\mathfrak{a}$ such that $\alpha(H)=1$. With respect to the inner product, induced by (2.1) on $\mathfrak{a}$, $H$ has norm 1.
We define 
\begin{equation}\label{a}
A^{+}:=\{\exp(tH)\colon t\in\R^{+}\};
\end{equation}
\begin{equation}\label{ro}
\rho:=\frac{1}{2}\dim(\mathfrak{g}_\alpha)\alpha;
\end{equation}
\begin{equation}\label{rom}
\rho_{\mathfrak{m}}:=\frac{1}{2}\sum_{\alpha\in \Delta^+(\mathfrak{m}_{\C},\mathfrak{b})}\alpha.
\end{equation}

The inclusion $i\colon M\hookrightarrow K$ induces the restriction map $i^{*}\colon R(K)\rightarrow R(M)$, where 
$R(K),R(M)$ are the representation rings over $\Z$ of $K$ and $M$, respectively.
Let $\widehat{K},\widehat{M}$ be the sets of equivalent classes of 
irreducible unitary representations of $K$ and $M$, respectively. For
the highest weight $\nu_{\tau}$ of $\tau\in\widehat{K}$, we have $\nu_{\tau}=(\nu_{1},\ldots,\nu_{n})$,
where $\nu_{1}\geq\ldots\geq\nu_{n}$ and $\nu_{i},i=1,\ldots,n$ are all half integers
(i.e.,  $\nu_{i}=q_{i}+\frac{1}{2}$, $q_{i}\in\Z$). For the highest weight $\nu_{\sigma}$ of $\sigma\in\widehat{M}$,
we have
\begin{equation}\label{hw}
\nu_{\sigma}=(\nu_{1},\ldots,\nu_{n-1},\nu_{n}),
\end{equation}
where $\nu_{1}\geq\ldots\geq\nu_{n-1}\geq\lvert\nu_{n}\rvert$ and $\nu_{i},i=1,\ldots,n$ are all half integers (\cite[p. 20]{BO}).
We denote by $(s,S)$ be the spin representation of $K$
and by $(s^{+},S^{+})$, $(s^{-},S^{-})$ the  half spin representations of $M$
(\cite[p. 22]{dirac}).

\subsection{The Casimir element}

Let $Z_{i}$ be a basis of $\mathfrak{g}$ and let 
$Z^{j}$ be the basis of $\mathfrak{g}$, which is determined by
$\langle Z_{i}, Z^{j}  \rangle=\delta_{ij}$, where
$\langle \cdot, \cdot \rangle$ is as in {\eqref{ip}}.
Let $\mathcal{U}(\mathfrak{g}_{\C})$ be the universal enveloping algebra of $\mathfrak{g}_{\C}$
and let $Z(\mathfrak{g}_{\C})$ be its center.
Then, $\Omega\in Z(\mathfrak{g}_{\C})$ is given by
\begin{equation*}
\Omega=\sum_{i}Z_{i}Z^{i}.
\end{equation*}
By Kuga's Lemma (\cite[(6.9)]{MM}),
the Hodge Laplacian, acting on the space $\Lambda^{*}(G/K)$ of differential forms on $G/K$, 
coincides with $-R(\Omega)$, where
$R(\Omega)$ is the Casimir operator on $\Lambda^{*}(G/K)$,
induced by $\Omega$.

\subsection{Definition of the twisted dynamical zeta functions}
We consider the twisted Ruelle and Selberg zeta functions associated with the geodesic flow on the sphere vector bundle $S(X)$ of $X=\Gamma\backslash G/ K$. Since $K$ acts transitively on the unit vectors in $\mathfrak{p}$,
$S(\widetilde{X})$ can be represented by the  homogeneous space $G/M$. 
Therefore $S(X)=\Gamma\backslash G/M$.

We recall the Cartan decomposition $G=KA^{+}K$ of $G$,
where $A^{+}$ is as in {\eqref{a}}.
Then, every element $g\in G$ can be written as $g=ha_{+}k$, where $h,k\in K$ and $a_{+}=\exp(tH)$,
 for some $t\in \R^{+}$.
The positive real number $t$ equals $d(eK,gK)$,
where $d(\cdot,\cdot)$ denotes the geodesic distance on $\widetilde{X}$.
It is a well known fact that there is a 1-1 correspondence between the closed geodesics on a manifold $X$ with negative sectional curvature 
and the non-trivial conjugacy classes of the fundamental group $\pi_{1}(X)$ of $X$ (\cite{GKM}). 

The hyperbolic elements of $\Gamma$ can be realized as the semisimple elements of this group, i.e., the diagonalizable elements of $\Gamma$.
Since $\Gamma$ is a cocompact subgroup of $G$, every element $\gamma\in\Gamma-\{e\}$ hyperbolic. 
We denote by $c_{\gamma}$ the closed geodesic on $X$,
 associated with the hyperbolic conjugacy class $[\gamma]$.
We denote by $l(\gamma)$ the length of $c_{\gamma}$.
Since $\Gamma$ is torsion-free, $l(\gamma)$ is always positive and therefore we can obtain an infimum for the length spectrum $\spec(\Gamma):=\{l(\gamma)\colon\gamma\in\Gamma\}$.
An element $\gamma\in\Gamma$ is called primitive if there exists no $n\in\N$ with $n>1$ and $\gamma_{0}\in\Gamma$ such that $\gamma=\gamma_{0}^{n}$.
A primitive element $\gamma_{0}\in\Gamma$ corresponds to a geodesic on $X$.
The prime geodesics correspond to the periodic orbits of minimal length.
Hence, if a hyperbolic element $\gamma$ in $\Gamma$ is generated by a primitive element $\gamma_{0}$, then there exists a $n_{\Gamma}(\gamma)\in \N$ such that $\gamma=\gamma_{0}^{n_{\Gamma}(\gamma)}$
and the corresponding closed geodesic is of length $l(\gamma)=n_{\Gamma}(\gamma)l(\gamma_{0})$.

We lift now the closed geodesic $c_{\gamma}$ to the universal covering $\widetilde{X}$.
For $\gamma\in\Gamma$, $l(\gamma):=\inf \{d(x,\gamma x)\colon x\in\widetilde{X}\}$,
or $l(\gamma)=\inf \{d(eK,g^{-1}\gamma gK)\colon g\in G\}$.
Hence, we see that the length of the closed geodesic $l(\gamma)$ depends only on $\gamma\in\Gamma$.
Let $\gamma \in \Gamma$, with $ \gamma\neq e$. Then, by \cite[Lemma 6.5]{Wa} there exist a $g\in G$, a $m_{\gamma} \in M$, and an 
$a_{\gamma} \in A^{+}$, such that $ g^{-1}\gamma g=m_{\gamma}a_{\gamma}$.
The element $m_{\gamma}$ is determined up to conjugacy in $M$, and the element $a_{\gamma}$ depends only on $\gamma$.
As in \cite[Section 3.1]{BO},
we consider the geodesic flow $\phi$ on $S(X)$, given by the map
$
\phi\colon\R\times S(X)\ni (t,\Gamma gM)\rightarrow \Gamma g\exp(-tH) M \in S(X)
$.
A closed orbit of $\phi$ is described by the set $c:=\{\Gamma g\exp(-tH) M\colon t\in\R\}$,
where $g\in G$ is such that $g^{-1}\gamma g:=m_{\gamma}a_{\gamma}\in MA^{+}$.
The Anosov property of the geodesic flow $\phi$ on $S(X)$ can be expressed by the following 
$d\phi$-invariant splitting of $TS(X)$
\begin{equation}\label{split}
 TS(X)=T^sS(X)\oplus T^cS(X)\oplus T^uS(X),
\end{equation}
where $T^sS(X)$ consists of vectors that shrink exponentially, $T^uS(X)$ 
consists of vectors that expand exponentially,
and $T^{c}S(X)$ is the one dimensional subspace of vectors tangent to the flow, 
with respect to the Riemannian metric, as $t\rightarrow\infty$.
The spitting in {\eqref{split}} corresponds to the splitting
\begin{equation*}
 TS(X)=\Gamma\backslash G\times_{\Ad}(\overline{\mathfrak{n}}\oplus \mathfrak{a} \oplus \mathfrak{n}),
\end{equation*}
where $\Ad$ denotes the adjoint action of $\Ad(\exp(-tH))$ on $\overline{\mathfrak{n}}, \mathfrak{a},\mathfrak{n}$, 
and $\overline{\mathfrak{n}}=\theta \mathfrak{n}$ is the sum of the negative root spaces of
the system $(\mathfrak{g},\mathfrak{a})$.

\begin{defi}
Let $\chi\colon\Gamma\rightarrow \GL(V_{\chi})$ be a finite-dimensional, complex
representation of $\Gamma$ and $\sigma\in \widehat{M}$.
The twisted Selberg zeta function $Z(s;\sigma,\chi)$ is defined  by the infinite product
\begin{equation*}
Z(s;\sigma,\chi):=\prod_{\substack{[\gamma]\neq{e}\\ [\gamma]\prim}} \prod_{k=0}^{\infty}\det\big(\Id-(\chi(\gamma)\otimes\sigma(m_\gamma)\otimes S^k(\Ad(m_\gamma a_\gamma)|_{\overline{\mathfrak{n}}})) e^{-(s+|\rho|)\lvert l(\gamma)}\big),
\end{equation*}
where $s\in \C$,
$S^k(\Ad(m_\gamma a_\gamma)_{\overline{\mathfrak{n}}})$ denotes the $k$-th
symmetric power of the adjoint map $\Ad(m_\gamma a_\gamma)$ restricted to $\mathfrak{\overline{n}}$ and $\rho$ is as in {\eqref{ro}}.
\end{defi}
By \cite[Proposition 3.4]{Spilioti2018}, 
there exists a positive constant $c$ such that 
the twisted Selberg zeta function $Z(s;\sigma,\chi)$ converges 
absolutely and uniformly on compact subsets of the half-plane $\RE(s)>c$ .

\begin{defi} Let $\chi\colon\Gamma\rightarrow \GL(V_{\chi})$ be a finite-dimensional, complex
representation of $\Gamma$ and $\sigma\in \widehat{M}$.
The twisted Ruelle zeta function $ R(s;\sigma,\chi)$ is defined by the infinite product
\begin{equation*}
 R(s;\sigma,\chi):=\prod_{\substack{[\gamma]\neq{e}\\ [\gamma]\prim}}\det\big(\Id-(\chi(\gamma)\otimes\sigma(m_{\gamma}))e^{-sl(\gamma)}\big)^{(-1)^{d-1}},
\end{equation*}
where $s\in\C$.
\end{defi}
By \cite[Proposition 3.5]{Spilioti2018}, 
there exists a positive constant $r$ such that 
the twisted Selberg zeta function $R(s;\sigma,\chi)$ converges 
absolutely and uniformly on compact subsets of the half-plane $\RE(s)>r$ .

The twisted dynamical zeta functions associated with an arbitrary representation
of $\Gamma$ has been studied in \cite{Spilioti2018} and \cite{spilioti2020functional}.
Specifically, by \cite[Theorem 1.1 and Theorem 1.2]{Spilioti2018},
the twisted dynamical zeta functions admit a meromorphic continuation to 
the whole complex plane.

\section{Harmonic analysis on symmetric spaces}
\subsection{Twisted Bochner-Laplace operator}

In this section, we define the twisted Bochner-Laplace operator as it is introduced  in 
\cite[Section 4]{M1}. In addition, we describe this operator in the locally symmetric space setting.

Let $E_{0}\rightarrow X$ be a complex vector bundle with covariant derivative $\nabla$. We define the second covariant derivative $\nabla^2$ by
\begin{equation*}
\nabla^2_{V,W}:=\nabla_{V}\nabla_{W}-\nabla_{\nabla^{LC}_{V}W}, 
\end{equation*}
where $V,W\in C^{\infty}(X,TX)$ and $\nabla^{LC}$ denotes the Levi-Civita connection on $TX$.
We define the connection Laplacian $\Delta_{E_{0}}$ to be the negative of the trace of the second covariant derivative, i.e., 
\begin{equation*}
 \Delta_{E_{0}}:=-\Tr\nabla^2.
\end{equation*}
If $E_{0}$ is equipped with a Hermitian metric compatible with the connection $\nabla$,
then by \cite[p. 154]{LM}, the connection Laplacian is equal to the Bochner-Laplace operator, i.e.,
\begin{equation*}
 \Delta_{E_{0}}=\nabla^{*}\nabla.
\end{equation*}
In terms of a local orthonormal frame field $(e_{1},\ldots,e_{d})$ of $T_{x}X$, for $x\in X$,
the connection Laplacian is given by
\begin{equation*}
\Delta_{E_{0}}=-\sum_{j=1}^{d}\nabla^ 2{_{e_j,e_j}}.
\end{equation*}
The principal symbol of $\Delta_{E_{0}}$ equals 
\begin{equation*}
 \sigma_{\Delta_{E_{0}}}(x,\xi)
 =\lVert \xi \rVert^ {2} \Id_{({E_{0}})_{x}},
\end{equation*}
where $x\in X$ and $\xi\in\T_{x}^{*}X$.
$\Delta_{E_{0}}$ acts in $L^{2}(X,E_{0})$ with domain $C^{\infty}(X,E_{0})$. 
The operator $\Delta_{E_{0}}\colon C^{\infty}(X,E_{0})\circlearrowright$ is a second order, elliptic, formally self-adjoint differential operator.

Let $\chi\colon\Gamma\rightarrow \GL(V_{\chi})$ be a finite-dimensional,
complex representation of $\Gamma$. 
Let $E_{\chi}\rightarrow X$ be the associated flat vector bundle 
over $X$, equipped with a flat connection $\nabla^{E_{\chi}}$.
We specialize to the twisted case  $E=E_{0}\otimes E_{\chi}$, where $E_{0}\rightarrow X$
is a complex vector bundle equipped with a connection $\nabla^{E_{0}}$ and 
a metric, which is compatible with this connection.
Let $\nabla^{E}=\nabla^{E_{0}\otimes E_{\chi}}$ be the product connection, defined by
\begin{equation*}
 \nabla^{E_{0}\otimes E_{\chi}}:=\nabla^{E_{0}}\otimes 1+1\otimes\nabla^{E_{\chi}}.
\end{equation*}
 We define the operator $\Delta_{E_{0},\chi}^\sharp$ by 
\begin{equation}\label{sharp}
 \Delta_{E_{0},\chi}^{\sharp}:=-\Tr\big((\nabla^{E_{0}\otimes E_{\chi}})^2\big).
\end{equation}
We choose a Hermitian metric in $E_{\chi}$.
Then, $\Delta_{E_{0},\chi}^{\sharp}$ acts on $L^{2}(X,E_{0}\otimes E_{\chi})$. However, it is not a formally self-adjoint operator in general.
We want to describe this operator locally. Following the analysis in \cite{M1}, we consider an open subset $U$ of $X$ such that $E_{\chi}\lvert_{U}$ is trivial.
Then, $E_{0}\otimes E_{\chi}\lvert_{U}$ is isomorphic to the direct sum of $m$-copies of $E_{0}\lvert_{U}$, i.e., 
\begin{equation*}
 (E_{0}\otimes E_{\chi})\lvert_{U}\cong\oplus_{i=1}^{m}E_{0}\lvert_{U},
\end{equation*}
where $m:=\rank(E_\chi)=\dim V_{\chi}$.
Let $(e_{i}),i=1,\ldots,m$ be any basis of flat sections of $E_{\chi}\lvert_{U}$.
Then, each $\phi \in C^{\infty}(U,(E_{0}\otimes E_{\chi})\lvert_{U})$ can be written as
\begin{equation*}
\phi=\sum_{i=1}^{m}\phi_{i} \otimes e_{i},
\end{equation*}
where $\phi_{i}\in C^{\infty}(U, E_{0}\lvert_{U}), i=1,\ldots,m$.
The product connection is given by
\begin{equation*}
  \nabla_{Y}^{E_{0}\otimes E_{\chi}}\phi=\sum_{i=1}^{m}(\nabla_{Y}^{E_{0}})(\phi_{i})\otimes e_{i},
\end{equation*}
where  $Y\in C^{\infty}(X,TX)$.
The local expression above is independent of the choice of the base of flat sections of $E_{\chi}\vert_{U}$, since the transition maps comparing flat sections are constant.
By {\eqref{sharp}}, we obtain the twisted Bochner-Laplace operator acting on $C^{\infty}(X,E_{0}\otimes E_{\chi})$ 
given by
\begin{equation}\label{sharploc}
 \Delta_{E_{0},\chi}^{\sharp}\phi=\sum_{i=1}^{m}(\Delta_{E_{0}}\phi_{i})\otimes e_{i},
\end{equation}
where $\Delta_{E_{0}}$ denotes the Bochner-Laplace operator $\Delta^{E_{0}}=(\nabla^{E_{0}})^*\nabla^{E_{0}}$
associated to the connection $\nabla^{E_{0}}$. 
Let now $\widetilde{E}_{0}, \widetilde{E}_{\chi}$ be the pullbacks to 
$\widetilde{X}$ of $E_{0},E_{\chi}$, respectively. 
Then,
\begin{equation*}
 \widetilde{E}_{\chi}\cong \widetilde{X}\times V_{\chi},
\end{equation*}
and
\begin{equation}\label{iso}
 C^{\infty}(\widetilde{X}, \widetilde{E}_{0}\otimes \widetilde{E}_{\chi})\cong  C^{\infty}(\widetilde{X}, \widetilde{E}_{0})\otimes V_{\chi}.
\end{equation}\label{univ}
With respect to the isomorphism {\eqref{iso}}, it follows from {\eqref{sharploc}} that the lift of $\Delta_{E_{0},\chi}^{\sharp}$ to $\widetilde{X}$ takes the form
\begin{equation}
 \widetilde{\Delta}^{\sharp}_{E_{0},\chi}=\widetilde{\Delta}_{E_{0}}\otimes \Id_{V_{\chi}},
\end{equation}
where $\widetilde{\Delta}_{E_{0}}$ is the lift of $\Delta_{E_{0}}$ to $\widetilde{X}$.
By {\eqref{sharploc}}, $\Delta_{E_{0},\chi}^{\sharp}$ has principal symbol 
\begin{equation*}
 \sigma_{\Delta_{E_{0},\chi}^\sharp}(x,\xi)=\lVert \xi \rVert^ {2}_{x} \Id_{({E_{0}\otimes E_{\chi})_{x}}}, 
 \quad x\in X, \xi\in\T_{x}^{*}X.
\end{equation*}
Hence,  it has nice spectral properties, i.e.,
its spectrum is discrete and contained in a translate of a positive cone $C\subset \C$ such that $\R^{+}\subset C$ (\cite[Lemma 8.6]{Spilioti2018}).

We specialize now the twisted Bochner-Laplace operator $\Delta_{E_{0},\chi}^{\sharp}$ to the case of the operator $\Delta^{\sharp}_{\tau,\chi}$
acting on smooth sections of the twisted vector bundle
$E_{\tau}\otimes E_{\chi}$. Here, $E_{\tau}$
is the locally homogeneous vector bundle
associated with a finite-dimensional, unitary representation $\tau$ of $K$. The keypoint is that when we consider the lift $\widetilde{\Delta}^{\sharp}_{\tau,\chi}$
of the twisted Bochner-Laplace operator  $\Delta^{\sharp}_{\tau,\chi}$  to the universal covering
\(\widetilde{X}\),
it acts as the identity operator on $V_{\chi}$.
By {\eqref{univ}}, we have
\begin{equation*}
 \widetilde{\Delta}^{\sharp}_{\tau,\chi}=\widetilde{\Delta}_{\tau}\otimes \Id_{V_{\chi}},
\end{equation*}
where $\widetilde{\Delta}_\tau$ is the lift to $\widetilde{X}$ of the Bochner-Laplace operator $\Delta_{\tau}$, associated with the
representation $\tau$ of $K$.

\subsection{The twisted operator $A_{\chi}^{\sharp}(\sigma)$}

In this section, we define the twisted operators $A_{\chi}^{\sharp}(\sigma)$, 
associated with $\sigma\in\widehat{M}$ and representations 
$\chi$ of $\Gamma$,  acting on smooth sections of twisted vector bundles.
These operators are first introduced in \cite{BO}.
For more details we refer the reader to \cite[Section 5]{Spilioti2018}.

We define the restricted Weyl group as the quotient $ W_{A}:=M'/M$.
Recall from Section 7.1 that $M'=\Norm_{K}(A)$ is the normalizer of $A$ in $K$.
Then, $W_{A}$ has order 2. Let $w\in W_{A}$ be the non-trivial element of $W_{A}$, and 
$m_{w}$ a representative of $w$ in $M'$.
The action of $W_{A}$ on $\widehat{M}$ is defined by
$(w\sigma)(m):=\sigma(m_{w}^{-1}mm_{w})$,
$m\in M, \sigma\in\widehat{M}$.
Following the proof of Proposition 1.1 in \cite{BO} (see also \cite[Proposition 2.3]{Pf}), 
there exist unique integers
$m_{\tau}(\sigma)\in\{-1,0,1\}$, which are equal to zero except for finitely many $\tau\in \widehat{K}$,
such that there exist unique integers
$m_{\tau}(\sigma)\in\{-1,0,1\}$, which are equal to zero except for finitely many $\tau\in \widehat{K}$,
such that,
\begin{itemize}
 \item 
  if $\sigma$ is Weyl invariant, 
$
 \sigma=\sum_{\tau\in\widehat{K}}m_{\tau}(\sigma)i^{*}(\tau);
$
 \item
 if $\sigma$ is non-Weyl invariant, 
$
 \sigma+w\sigma=\sum_{\tau\in\widehat{K}}m_{\tau}(\sigma)i^{*}(\tau).
$
\end{itemize}
We define a locally homogeneous vector bundle $E(\sigma)$ associated to $\sigma$ by
\begin{equation*}
 E(\sigma):=\bigoplus_{\substack{\tau\in\widehat{K}\\m_{\tau}(\sigma)\neq 0}}E_{\tau},
\end{equation*}
where $E_{\tau}$ is the locally homogeneous vector bundle associated with $\tau\in\widehat{K}$.
The vector bundle $E(\sigma)$ has a grading 
\begin{equation}\label{gr}
E(\sigma)=E(\sigma)^{+}\oplus E(\sigma)^{-}.
\end{equation}
This grading is defined exactly by the positive or negative sign of $m_{\tau}(\sigma)$.
Let $\widetilde{E}(\sigma)$ be the pullback of $E(\sigma)$ to $\widetilde{X}$.
Then,
\begin{equation*}
 \widetilde{E}(\sigma)=\bigoplus_{\substack{\tau\in\widehat{K}\\m_{\tau}(\sigma)\neq0}}\widetilde{E}_{\tau}.
\end{equation*}
We assume now that $\tau\in\widehat{K}$ is irreducible. 
Recall that 
\(
 \widetilde{\Delta}_{\tau}=-R(\Omega)+\lambda_{\tau}\Id.
\)
(\cite[Proposition 1.1]{Mi1}, see also \cite[(5.4)]{Spilioti2018})
We put
\begin{equation*}
 \widetilde{A}_{\tau}:=\widetilde{\Delta}_{\tau}-\lambda_{\tau}\Id.
\end{equation*}
Hence, the operator $\widetilde{A}_{\tau}$ acts like $-R(\Omega)$ on the space of smooth sections of $\widetilde{E}_{\tau}$. 
It is an elliptic, formally self-adjoint operator of second order. By \cite{Ch}, it is an essentially self-adjoint operator.
Its self-adjoint extension will be also denoted by $\widetilde{A}_{\tau}$.
We get then the operator $\widetilde{A}^{\sharp}_{\tau,\chi}$ acting on 
the space $C^{\infty}(\widetilde{X},\widetilde{E}_{\tau}\otimes\widetilde{E}_{\chi})$,
defined by
\begin{equation*}
 \widetilde{A}^{\sharp}_{\tau,\chi}=\widetilde{A}_{\tau}\otimes\Id_{V_{\chi}}.
\end{equation*}
We  put
\begin{equation}\label{constant}
 c(\sigma):=-\lvert \rho \rvert^{2}-\lvert\rho_{m}\rvert^{2}+\lvert \nu_{\sigma}+\rho_{m}\rvert^{2},
\end{equation}
where $\nu_{\sigma}$ is as in {\eqref{hw}}
and $\rho,\rho_{m}$ are defined by {\eqref{ro}} and {\eqref{rom}}, respectively.
We define the operator $A_{\chi}^{\sharp}(\sigma)$ acting on $C^{\infty}(X,E(\sigma)\otimes E_{\chi})$ by
\begin{equation}\label{op}
 A_{\chi}^{\sharp}(\sigma):=\bigoplus_{m_{\tau}(\sigma)\neq 0}A_{\tau,\chi}^{\sharp}+c(\sigma).
\end{equation}
The operator  $A_{\chi}^{\sharp}(\sigma)$ preserves the grading and 
it is a non-self-adjoint, elliptic operator of order two.

\section{The determinant formula}

In  \cite{spilioti2020functional}, a determinant formula has been proved, which gives an interpretation of the twisted Ruelle zeta function in terms of graded, regularized determinants of twisted, non-self-adjoint, elliptic differential operators. 
We recall here some facts and definitions from \cite{spilioti2020functional}.

Let $\sigma_{p}$ be the standard representation of $M$ in $\Lambda^{p}\R^{d-1}\otimes\C$. Let $\sigma,\sigma'\in\widehat{M}$. We denote by $[\sigma_{p}\otimes\sigma:\sigma']$
the multiplicity of $\sigma'$ in $\sigma_{p}\otimes\sigma$.
We distinguish again two cases for $\sigma'\in\widehat{M}$. 
\begin{itemize}
 \item  {\bf case (a)}: \textit{$\sigma'$ is invariant under the action of the restricted Weyl group $W_{A}$.}
Then, $i^{*}(\tau)=\sigma'$, where $\tau\in R(K)$.
 \item {\bf case (b)}: \textit{$\sigma'$ is not invariant under the action of the restricted Weyl group $W_{A}$.}
Then, $i^{*}(\tau)=\sigma'+w\sigma'$, where $\tau\in R(K)$.
\end{itemize}
We define the operator
\begin{equation}\label{opnonirr}
 A_{\chi}^{\sharp}(\sigma_{p}\otimes\sigma):=\bigoplus_{[\sigma']\in\widehat{M}/W_{A}}\bigoplus_{i=1}^{[\sigma_{p}\otimes\sigma:\sigma']}A^{\sharp}_{\chi}(\sigma'),
\end{equation}
acting on the space $C^{\infty}(X,E(\sigma')\otimes E_{\chi})$, where $E(\sigma')$ is the vector bundle 
over $X$, constructed as in \cite[p. 175]{Spilioti2018}.

Let $\alpha$ be the unique positive root of $(\mathfrak{g},\mathfrak{a})$. Let $H\in\mathfrak{a}$ such that 
$\alpha(H)=1$. The character $\lambda\equiv\lambda_{p}$ of $A$ is defined by $\lambda\equiv\lambda_{p}(a)=e^{p\alpha(\log a)}$. Then,
we can identify $\lambda$ with $p$.
By \cite[Proposition 7.9, case (a)]{spilioti2020functional}, we have the following proposition.
\begin{prop}
The Ruelle zeta function has the representation 
\begin{align}\label{det1}
R(s;\sigma,\chi)=\notag&\prod_{p=0}^{d-1}{{\det}_{\gr}(A^{\sharp}_{\chi}(\sigma_{p}\otimes\sigma)+(s+\vert\rho\vert-p)^{2})}^{(-1)^{p}}\\
&\cdot\exp\bigg((-1)^{\frac{d-1}{2}+1}\pi(d+1)\dim(V_{\sigma})\dim (V_{\chi})\frac{\Vol(X)}{\Vol(S^{d})}s\bigg).
\end{align}
\end{prop}

\begin{rmrk}
In the determinant formula (9.2), the regularized determinants 
are graded, with respect to the grading 
\eqref{gr} of the locally
homogeneous vector bundle $E(\sigma)$ over $X$.
\end{rmrk}

\begin{thm}\label{detth}
The Ruelle zeta function has the representation
\begin{align}\label{det2}
R(s;\chi)=\notag&\prod_{k=0}^{d-1}\prod_{p=k}^{d-1}{\det\big(\Delta^{\sharp}_{\chi,k}+s(s+2(\vert\rho\vert-p))\big)}^{(-1)^{p}}\\
&\cdot\exp\bigg((-1)^{\frac{d-1}{2}+1}\pi(d+1)\dim (V_{\chi})\frac{\Vol(X)}{\Vol(S^{d})}s\bigg).
\end{align}
Let $d_{\chi,k}:=\dim\Ker(\Delta_{\chi,k})$. Then, the singularity of the Ruelle zeta function at $s=0$ is of order
\begin{equation}\label{order}
\sum_{k=0}^{(d-1)/2}(d+1-2k)(-1)^{k}d_{\chi,k}.
\end{equation}
\end{thm}
\begin{proof}
We denote by $r_{p}$ be the $p$-th exterior power of the standard representation 
of $K$. 
This notation coincides with the notation used in 
\cite[p. 23--24]{BO}.
Then, for $p=0,\ldots, d-1$, we have $i^{*}(r_{p})=\sigma_{p}+\sigma_{p-1}$.
Let $k_{p}\in R(K)$, defined by
\begin{equation*}
 k_{p}:=\sum_{k=0}^{p}(-1)^{k}r_{p-k}.
\end{equation*}
Then, for $p=0,1,\ldots, d-1$,
one has $\sigma_{p}=i^{*}(k_{p})$.
Let $E_{k_{p}}$, $E_{r_{k}}$ be the 
locally homogeneous vector bundles over $X$, associated with
$k_{p}$ and $r_{k}$, correspondingly.
We consider $\sigma$ being trivial. 
By \eqref{constant}, $c(\sigma_{p})=(\vert\rho\vert-p)^{2}$.
By \eqref{op} and \eqref{opnonirr},
the operators $A^{\sharp}_{\chi}(\sigma_{p})+(\vert\rho\vert-p)^{2}$ and 
$-R(\Omega)$ on $C^{\infty}(X, E_{k_{p}}\otimes V_{\chi} )$ coincide.
On the other hand, by Kuga's Lemma (Subsection 2.1.3),
the Casimir operator $R(\Omega)$ acts as the negative Hodge Laplacian 
$-\Delta^{\sharp}_{\chi,k}$  on $C^{\infty}(X, E_{r_{k}}\otimes V_{\chi})$.
Hence, by  \eqref{det1}, we get
\begin{align*}
R(s;\chi)=\prod_{p=0}^{d-1}&{{\det}_{\gr}(A^{\sharp}_{\chi}(\sigma_{p})+(s+\vert\rho\vert-p)^{2})}^{(-1)^{p}}\\
&\cdot\exp\bigg((-1)^{\frac{d-1}{2}+1}\pi(d+1)\dim (V_{\chi})\frac{\Vol(X)}{\Vol(S^{d})}s\bigg)\\
=\prod_{k=0}^{d-1}&\prod_{p=k}^{d-1}{\det\big(\Delta^{\sharp}_{\chi,k}+s(s+2(\vert\rho\vert-p))\big)}^{(-1)^{p}}\\
&\cdot\exp\bigg((-1)^{\frac{d-1}{2}+1}\pi(d+1)\dim (V_{\chi})\frac{\Vol(X)}{\Vol(S^{d})}s\bigg).
\end{align*}
For the order of the singualrity of the Ruelle zeta function at $s=0$, 
\eqref{order} follows from  Poincar\'e duality on differential forms.
\end{proof}

\begin{rmrk}
We note here that for $p=\frac{d-1}{2}$, $\sigma_{p}$ is not irreducible. 
In fact, $\sigma_{\frac{d-1}{2}}$ decomposes into $\sigma_{+}+\sigma_{-}$,
and $\sigma_{\pm}=w\sigma_{\mp}$.
Then, we use \eqref{opnonirr} to define $A^{\sharp}_{\chi}(\sigma_{\frac{d-1}{2}})$.
By \cite[Lemma 1.4]{millson1978closed}, $c(\sigma_{+})=c(\sigma_{-})=0$.
\end{rmrk}

\begin{rmrk}
In the case of unitary representations, which are acyclic, we can obtain the triviality 
of the kernels of the Hodge Laplacian $\Delta_{\chi,k}$,
since by Hodge theorem we have
\begin{equation*}
H^{k}(X,E_{\chi})\cong \Ker(\Delta_{\chi,k})=\{0\},
\end{equation*}
where $H^{k}(X,E_{\chi})$ is the cohomology with coefficients in the local system 
defined by $\chi$. 
In the present case, there is no such an isomorphism and
one can find non-injective Laplacians associated with acyclic representations of $\Gamma$.
\end{rmrk}

\section{Ruelle zeta function and refined analytic torsion}

 Let $V$ be as in in Subsection \ref{Vnei}, i.e., $V$ is an open neighbourhood of the set $\Rep^{u}_{0}(\pi_{1}(X),\C^{n})$ of acyclic and unitary representations of $\pi_{1}(X)=\Gamma$, such that,
for all $\chi\in V$, $B_{\chi}$ is bijective. Then, we have the following theorem. 

\begin{thm}\label{main}
Let $\chi\in V$. Then, the Ruelle zeta function $R(s;\chi)$ is regular at $s=0$ and is equal to the complex Cappell-Miller torsion, 
\begin{equation}\label{ruzero}
R(0;\chi)=\tau_{\chi}.
\end{equation}
\end{thm}
\begin{proof}
For $\chi\in V$, the operator $B_{\chi,k}^{2}$ coincides with the flat Laplacian $\Delta^{\sharp}_{\chi,k}$ on 
$\Lambda^{k}(X,E_{\chi})$ and it is injective. 
Hence, by Theorem 5.3, the Ruelle zeta function $R(s;\chi)$
is regular at zero and  in addition, by {\eqref{det2}}, 
\begin{equation}\label{ruzeroproof}
R(0;\chi)=\prod_{k=0}^{d-1}{\det(\Delta^{\sharp}_{\chi,k})}^{(d-k)(-1)^{k}}
= \prod_{k=1}^{d}{\det(\Delta^{\sharp}_{\chi,k})}^{k(-1)^{k-1}}.
\end{equation}
Hence, \eqref{ruzero} follows by \eqref{cmchi} and \eqref{ruzeroproof}.
\end{proof}

\begin{coro}\label{maincoro}
 Let $\chi\in V$. Then the Ruelle zeta function $R(s;\chi)$ is regular at $s=0$ and is related to the refined analytic torsion  $T_{\chi}$ by
 \begin{equation}\label{rurefined}
  R(0;\chi)=T_{\chi}^{2} e^{2\pi i(\eta(B^{\ev}_{\chi})-\rank(E_{\chi})\eta_{\tr})}.
 \end{equation}
\end{coro}
\begin{proof}
The assertion follows from \eqref{ruzero} and \eqref{comparison}.
\end{proof}

Let now $\chi\in \Rep^{u}_{0}(\pi_{1}(X),\C^{n})$. Then, by \cite[Theorem 4.8]{BO}, the Ruelle zeta function $R(s,\chi)$ associated with $\chi$ is regular at $s=0$. Moreover, the Cappell-Miller torsion coincides with the square of the Ray-Singer real-valued torsion $T^{RS}_{\chi}$. 
On the other hand, in this case, the refined analytic torsion is still a complex number, given  by
\begin{equation*}
 T_{\chi}= e^{\xi_{\chi}} e^{-i\pi(\eta(B^{\ev}_{\chi})-\rank(E_{\chi})\eta_{tr})}.
\end{equation*}
Here, we denote $\xi_{\chi}=\xi(\chi, g, \theta)$.
If $\chi$ is unitary and acyclic, the term $e^{\xi_{\chi}}$ coincides with the Ray-Singer torsion
$T^{RS}_{\chi}$, and the eta invariant $\eta(\chi,g)$ is real valued. 
Hence,
\begin{equation}\label{ruabs1}
 \lvert T_{\chi} \rvert= e^{\xi_{\chi}}= T^{RS}_{\chi}
\end{equation}
Therefore, \eqref{ruzero}, \eqref{rurefined} and \eqref{ruabs1}, yields the following classical result
(see \cite[Theorem 4.8]{BO}).
\begin{coro}
 Let $\chi\in \Rep^{u}_{0}(\pi_{1}(X),\C^{n})$. Then, 
 the Ruelle zeta function $R(s;\chi)$ associated with $\chi$ is regular at $s=0$ and 
 \begin{equation*}
  R(0;\chi)=\tau_{\chi}=\lvert \T_{\chi}\rvert^{2}=(T^{RS}_{\chi})^{2}.
 \end{equation*}
\end{coro}

\begin{coro}
 Let $\chi\in V$. Then, the Ruelle zeta function $R(s;\chi)$ is regular at $s=0$ and 
 \begin{equation}\label{ruabs2}
  \lvert R(0;\chi) \rvert=\lvert \tau_{\chi} \rvert=(T^{RS}_{\chi})^{2} 
 \end{equation}
\end{coro}
\begin{proof}
 By \cite[Theorem 8.2]{BK2}, we have
 \begin{equation}\label{rexi}
  e^{\RE(\xi_{\chi})}=T^{RS}_{\chi}.
 \end{equation}
Hence, \eqref{ruabs2} follows from \eqref{ruzero}, \eqref{rurefined} and \eqref{rexi}.
\end{proof}

Using the alternative definition of the refined analytic torsion as in Section 8, we obtain the following
corollary.
\begin{coro}
 Let $\chi\in V$ and $n=4k$, $k\in\Z$. Then, the Ruelle zeta function $R(s;\chi)$ is regular at $s=0$ and 
  \begin{equation}
  R(0;\chi)={T'_{\chi}}^{2} e^{2\pi i\big(\eta(B^{\ev}_{\chi})-\frac{\rank(E_{\chi})}{2} \int_{N}L(p)\big)}
 \end{equation}
\end{coro}

\begin{proof}
It follows from \eqref{ruzeroproof}, the definition of $\xi$ in \eqref{xi}
and the expression of $T'(\nabla)$ in \eqref{altref}.
Here, we denote $T'_{\chi}=T'(\nabla_{\chi})$.
\end{proof}

\section{Appendix: Refined analytic torsion as an element of the determinant line}

For a general non-unitary representation, the definition of the refined analytic torsion differs from \ref{defref} and is not a complex number any more. It is an element of the determinant line of the cohomology. We recall here the definition from  \cite{BK3}. Let $\mathbf{k}$ be a field of characteristic zero. Let $V$ be a $\mathbf{k}$-vector space of dimension $n$. We define the determinant line of $V$ by $\det(V^{*}):=\Lambda^{n}V$, where $\Lambda^{n}V$ denotes the $n$-th exterior power of $V$. We set $\det(0):=\mathbf{k}$. 
If $L$ is a $\mathbf{k}$-line (one-dimensional vector space), we define the dual line $L^{-1}$ by $L^{-1}:=\text{Hom}_{\mathbf{k}}(L,\mathbf{k})$. For $l\in L$, we denote by $l^{-1}\in L^{-1}$ the unique $\mathbf{k}$-linear map $L\rightarrow\mathbf{k}$, such that $l^{-1}(l)=1$.
For a graded $\mathbf{k}$-vector space 
$V^{*}=V^{0}\oplus V^{1} \cdots\oplus V^{d}$, we define
\begin{equation*}
\det (V^{*}):= \bigotimes_{j=0}^{d}\det(V^{j})^{(-1)^{j}}.
\end{equation*}
Let $(C^{*},\partial)$,
\begin{equation*}
0\rightarrow C^{0} 
\overset{\partial}{\longrightarrow} C^{1}
\overset{\partial}{\longrightarrow}\cdots 
\overset{\partial}{\longrightarrow}C^{d}\rightarrow 0,
\end{equation*}
be a complex of finite-dimensional $\mathbf{k}$-vector spaces.
We call the integer $d$ the length of the complex $(C^{*},\partial)$. 
We denote by 
\begin{equation*}
H^{*}(\partial)=\bigoplus_{i=0}^{d}H^{i}(\partial)
\end{equation*}
the cohomology of $(C^{*},\partial)$.
We set
\begin{align*}
 &\det (C^{*}):= \bigotimes_{j=0}^{d}\det(C^{j})^{(-1)^{j}}\\
 &\det (H^{*}):= \bigotimes_{j=0}^{d}\det(H^{j}(\partial))^{(-1)^{j}}.
\end{align*}
We want to consider the determinant line of a direct sum
of two (or more) finite-dimensional $\mathbf{k}$-vector spaces.
Let $V,W$ be finite-dimensional $\mathbf{k}$-vector spaces
of dimension $\dim V=k$ and $\dim W=l$, correspondingly.
We define the canonical fusion isomorphism 
\begin{equation*}
 \mu_{V,W}\colon\det(V)\otimes \det(W)\rightarrow \det(V\oplus W)
\end{equation*}
by
\begin{align*}
  \mu_{V,W}\colon (v_{1}\wedge v_{2}\wedge\ldots\wedge v_{k})\otimes
  (w_{1}&\wedge w_{2}\wedge\ldots\wedge w_{l})\\
  &\mapsto v_{1}\wedge v_{2}\wedge\ldots\wedge v_{k}\wedge 
  w_{1}\wedge w_{2}\wedge\ldots\wedge w_{l},
\end{align*}
where $v_{i}\in V$ and $w_{j}\in W$.
For $v\in \det(V)$ and $w\in \det (W)$, we have
\begin{equation*}
 \mu_{V,W} (v\otimes w)= (-1)^{\dim V\dim W} \mu_{V,W} (w\otimes v).
\end{equation*}
We denote by 
\begin{equation*}
  \mu_{V,W}^{-1}\colon \det (V)^{-1}\otimes \det(W)^{-1}
  \rightarrow \det (V\oplus W)^{-1}
\end{equation*}
the transpose of the inverse of $\mu_{V,W}$. 
Then, it follows that for  $v\in \det(V)$ and $w\in \det (W)$
\begin{equation*}
 \mu_{V,W}^{-1}(v^{-1}\otimes w^{-1})= (\mu_{V,W} (v\otimes w))^{-1}.
\end{equation*}
Let now $V_{1},\ldots V_{r}$, finite-dimensional $\mathbf{k}$-vector spaces. We define an isomorphism
\begin{equation*}
   \mu_{V_{1},\ldots,V_{r}} \colon \det (V_{1})\otimes\ldots\otimes\det (V_{r})\rightarrow\det(V_{1}\oplus\ldots\oplus V_{r}).
\end{equation*}
For $j=1,\ldots, r-1$, it holds
\begin{align*}
  \mu_{V_{1},\ldots,V_{r}}= 
   \mu_{V_{1},\ldots,V_{j-1},V_{j}\oplus V_{j+1}, V_{j+2},\ldots,V_{r}}
  \circ (1\otimes\cdots 1\otimes \mu_{V_{j},V_{j+1}}\otimes 1\otimes 1).
\end{align*}
We fix now a direct sum decomposition 
\begin{equation*}
 C^{j}=B^{j}\oplus H^{j} \oplus A^{j}, \quad j=0,\ldots, d,
\end{equation*}
such that $B^{j}\oplus H^{j}=\Ker(\partial)\cap C^{j}$ and $B^{j}=\partial(C^{j-1})=\partial (A^{j-1)})$, for all $j$. Note that
$A^{d}={0}$. Set $A^{-1}=\{0\}$. Then, $H^{j}$ is naturally isomorphic to the cohomology $H^{j}(\partial)$ and $\partial$ defines an isomorphism $\partial\colon A^{j}\rightarrow B^{j+1}$.
For each $j=0,\ldots,d$, we fix $c_{j}\det (C^{j})$ and $a_{j}\in\det (A^{j})$. Let $\partial (a_{j})\det (B^{j+1})$ be the image
of $a_{j}$ under the map $\det(A^{j})\rightarrow \det(B^{j+1})$,
induced by the isomorphism $\partial\colon A^{j}\rightarrow B^{j+1}$. Then, for each $j=0,\cdots, d$, there is a unique element 
$h_{j}\det(H^{j})$ such that
\begin{equation*}
 c_{j}=\mu_{B^{j}, H^{j}, A^{j}}(\partial(a_{j-1})\otimes h_{j}\otimes a_{j}).
\end{equation*}
Let $\phi_{C^{*}}$ be the isomorphism
\begin{equation}\label{isocoh}
 \phi_{C^{*}}= \phi_{(C^{*},\partial)}:\det(C^{*})\rightarrow \det(H^{*}(\partial)),
\end{equation}
defined by
\begin{equation*}
 \phi_{C^{*}}:c_{0}\otimes c_{1}^{-1}\otimes\cdots\otimes c_{d}^{(-1)^{d}} \mapsto
(-1)^{\mathcal{N}(C^{*})} h_{0}\otimes h_{1}^{-1}\otimes\cdots\otimes h_{d}^{(-1)^{d}},
\end{equation*}
where
\begin{equation*}
 \mathcal{N}(C^{*}):=\frac{1}{2}\sum_{j=0}^{d}\dim A^{j}(\dim A^{j} +(-1)^{j+1}).
\end{equation*}
Then, $\phi_{C^{*}}$  is independent of the choices of $c_{j}$ and $a_{j}$.
\begin{rmrk}
The isomorphism $\phi_{C^{*}}$ is a sign refinement of the standard
construction by Milnor (\cite{milnor1966whitehead}). The sign factor $\mathcal{N}(C^{*})$ 
was introduced by Braverman and Kappeler in order to obtain various 
compatibility properties (see \cite[Remark 2.5, Lemma 2.7 and Proposition 5.6]{BK3}).
\end{rmrk}

\subsection{Refined torsion of a finite-dimensional complex with a chirality operator}

We consider now the chirality operator $\Gamma: C^{*}\rightarrow C^{*}$,
such that $\Gamma: C^{j}\rightarrow C^{d-j}$, for $j=0,\cdots,d$.
For $c_{j}\in \det(C^{j})$, $\Gamma c_{j}\in \det(C^{d-j})$ is the image of $c_{j}$ under the isomorphism $\det(C^{j})\rightarrow \det(C^{d-j})$,
induced by $\Gamma$.
Let $d=2r-1$ be an odd integer. 
We fix non-zero elements $c_{j}\in\det(C^{j})$, $j=0,\cdots, r-1$, and consider the element
\begin{align*}
c_{\Gamma}:= (-1)^{\mathcal{R}(C^{*})}c_{0}\otimes & c_{1}^{-1}\otimes \cdots c_{r-1}^{(-1)^{r-1}}\\ 
&\otimes(\Gamma c_{r-1})^{(-1)^{r}}
\otimes(\Gamma c_{r-2})^{(-1)^{r-1}}
\otimes\cdots \otimes  (\Gamma c_{0}^{-1})\notag
\end{align*}
of $\det(C^{*})$, where
\begin{equation*}
 \mathcal{R}(C^{*})=\frac{1}{2}\sum_{j=0}^{r-1}\dim C^{j}(\dim C^{j}+(-1)^{r+j}).
\end{equation*}
It follows from the definition of $c_{j}^{-1}$ that 
$c_{\Gamma}$ is independent of the choice of
$c_{j}$, $j=0,\ldots,r-1$.

\begin{defi}
We define the refined torsion $\rho_{\Gamma}$ of the pair $(C^{*},\Gamma)$ by
\begin{equation}\label{torsionfinite}
\rho_{\Gamma}=\rho_{C^{*},\Gamma}:=\phi_{C^{*}}(c_{\Gamma}),
\end{equation}
where $\phi_{C^{*}}$ is as in \eqref{isocoh}.
\end{defi}

We consider now the finite-dimensional analogue of the odd signature operator $B$, 
defined  in a Riemannian manifold-setting by \eqref{odd}. 
Let $B\colon C^{*}\rightarrow C^{*}$ be the signature operator, defined by
\begin{equation*}
 B:=\partial\Gamma+\Gamma\partial.
\end{equation*}
We consider its square,  $B^{2}=(\partial\Gamma)^{2}+(\Gamma\partial)^{2}$.
Let $I\subset[0,\infty)$.  Let $C^{j}_{I}\subset C^{j}$ be the span of the generalized eigenvectors of the restriction of $B^{2}$ to $C^{j}$, corresponding to eigenvalues $r$ with $\vert  r \vert\in I$.
Both $\Gamma$ and $\partial$ commute with $B$ and hence with $B^{2}$. 
Then, we have $\Gamma:C^{j}_{I}\rightarrow C^{d-j}_{I}$, and 
$\partial: C^{j}_{I}\rightarrow C^{j+1}_{I}$.
Hence, we obtain a subcomplex $C_{I}^{*}$ of $C^{*}$ and the restriction  $\Gamma_{I}$
of $\Gamma$ to $C_{I}^{*}$ is the chirality operator on this complex.
Let $\partial_{I}, B_{I}, B^{\ev}_{I}$ be the restriction to $C_{I}^{*}$ of $\partial, B, B^{\ev}$, 
correspondingly. 
Here, we denote $ B^{\ev}\colon C^{\ev}\rightarrow C^{\ev}$,
where $C^{\ev}:=\bigoplus_{j \ev} C^{j}$.
By Lemma 5.8 in \cite{BK3}, if $0 \notin I$, then the complex $(C^{*}_{I},\partial_{I})$
is acyclic.
For every $\lambda\geq 0$, we have
\begin{equation*}
 C^{*}=C^{*}_{[0,\lambda]}\oplus C^{*}_{(\lambda,\infty)}.
\end{equation*}
In particular,
\begin{align*}
 H^{*}_{(\lambda,\infty)}(\partial)=0, \quad
 H^{*}_{[0,\lambda]}(\partial)\cong H^{*}(\partial).
\end{align*}
Hence, there are canonical isomorphisms
\begin{align*}
 \Phi\colon\det(H^{*}_{(\lambda,\infty)}(\partial))\rightarrow\C, \quad
 \Psi\colon\det(H^{*}_{[0,\lambda]}(\partial))\rightarrow\det(H^{*}(\partial)).
\end{align*}
By \cite[Proposition 5.10]{BK3}, for each $\lambda\geq 0$,
\begin{equation*}
\rho_{\Gamma}={\det}_{\gr}(B^{\ev}_{(\lambda,\infty)})\rho_{\Gamma_{[0,\lambda]}},
\end{equation*}
where $\rho_{\Gamma_{[0,\lambda]}}$ is an element of $\det(H^{*}(\partial))$ via the canonical isomorphism $\Psi$.

\subsection{Refined analytic torsion as an element of the determinant line for compact odd-dimensional manifolds}

We consider now a compact, oriented, Riemannian manifold $(X,g)$ of odd dimension $d=2r-1$ and a complex vector bundle $E$ over $X$, endowed with a flat connection $\nabla$. Let $\Lambda^{k}(X,E)$ be the space of $k$-differential forms on $X$ with values in $E$. 
Let $B$ be the odd signature operator acting on $\Lambda^{k}(X,E)$ as in Definition \ref{oddsign}. Let $I\subset[0,\infty)$.
Let  $\Lambda^{k}_{I}(X,E)$ be the image of $\Lambda^{k}(X,E)$ under the spectral projection of $B^{2}$, corresponding to eigenvalues whose absolute value lie in $I$. 
If $I$ is bounded, then the subspace $\Lambda^{k}_{I}(X,E)$ is finite-dimensional 
(\cite[Section 6.10]{BK3}).
Let $B_{I}$, $\Gamma_{I}$ be the restrictions to $\Lambda^{k}_{I}(X,E)$ of $B$ and $\Gamma$, respectively.
For every $\lambda\geq0$, we have
\begin{equation*}
 \Lambda^{*}(X,E)=\Lambda^{*}_{[0,\lambda]}(X,E)\oplus \Lambda^{*}_{(\lambda,\infty)}(X,E).
\end{equation*}
Since for every $\lambda\geq0$ the complex $\Lambda^{*}_{(\lambda,\infty)}(X,E)$ is acyclic,
we get
\begin{equation*}
  H^{*}_{[0,\lambda]}(X,E)\cong H^{*}(X,E).
\end{equation*}
\begin{defi}
The refined analytic torsion $T=T(\nabla)$ is an element of $\det(H^{*}(X,E))$, defined by
\begin{equation*}
 T(\nabla)=\rho_{\Gamma_{[0,\lambda]}}{\det}_{\gr}(B^{\ev}_{(\lambda,\infty)})e^{i\pi \rank(E_{\chi})\eta_{tr}(g)},
\end{equation*}
where $\rho_{\Gamma_{[0,\lambda]}}$ is as in \eqref{torsionfinite}.
\end{defi}
By \cite[Proposition 7.8]{BK3} and \cite[Theorem 9.6]{BK3}, the refined analytic torsion is independent 
the choice of $\lambda\geq 0$, the choice of the Agmon angle $\theta\in(-\pi,0)$
for $B^{\ev}$, and the Riemannian metric $g$.

\bibliographystyle{amsplain}
\bibliography{ref}
\contact
\end{document}